\documentclass{cmslatex}
\usepackage[paperwidth=7in, paperheight=10in, margin=.875in]{geometry}
 \usepackage[backref,colorlinks,linkcolor=red,anchorcolor=green,citecolor=blue]{hyperref}
\usepackage{amsfonts,amssymb}
\usepackage{amsmath}
\usepackage{graphicx}
\usepackage{cite}
\usepackage{enumerate}
\renewcommand{\theequation}{\arabic{section}.\arabic{equation}}

\newcommand{\commentout}[1]{}
\newcommand{\R}{\mathbb{R}}

\newcommand{\SM}{\textcolor{black}}

\newcommand {\e}  {\varepsilon}

\newcommand {\sg} {\sigma}

\newcommand {\Chi} {{\bf \raise 2pt \hbox{$\chi$}} }

\newcommand {\f}   {\frac}
\newcommand {\p}   {\partial}
\newcommand{\fer}{\eqref}
\newcommand{\dis}{\displaystyle}

   \allowdisplaybreaks
\begin{document}
 \title{Selection and mutation in a shifting and fluctuating environment\thanks{Received date, and accepted date (The correct dates will be entered by the editor).}}

          %For each author, make a block with the following macros:

          \author{Susely Figueroa Iglesias\thanks{Institut de Math\'ematiques de Toulouse; UMR 5219, Universit\'e de Toulouse; CNRS, UPS IMT, F-31062 Toulouse Cedex 9, France; E-mail: Susely.Figueroa@math.univ-toulouse.fr}, \and Sepideh Mirrahimi\thanks{Institut de Math\'ematiques de Toulouse; UMR 5219, Universit\'e de Toulouse; CNRS, UPS IMT, F-31062 Toulouse Cedex 9, France; E-mail: Sepideh.Mirrahimi@math.univ-toulouse.fr} }

         \pagestyle{myheadings} \markboth{Adaptation in a time-varying environment}{S.FIGUEROA IGLESIAS \& S. MIRRAHIMI} \maketitle

          \begin{abstract}
               We study the evolutionary dynamics of a phenotypically structured population in a changing environment, where the environmental conditions vary with a linear trend but in an oscillatory manner. Such phenomena can be described by parabolic Lotka-Volterra type equations with non-local competition and a time dependent growth rate. We first study the long time behavior of the solution to this problem. Next, using an approach based on Hamilton-Jacobi equations we study asymptotically such long time solutions when the effects of the mutations are small. We prove that, as the effect of the mutations vanishes, the phenotypic density of the population concentrates on a single trait which varies linearly with time, while the size of the population oscillates periodically. In contrast with the case of an environment without linear shift, such dominant trait does not have the maximal growth rate in the averaged environment and there is a cost on the growth rate due to the {environmental} shift. We also provide an asymptotic expansion for the average size of the population and for the critical speed above which the population goes extinct, which is closely related to the derivation of an asymptotic expansion for the Floquet eigenvalue in terms of the diffusion rate. By mean of a biological example, this expansion allows to show that the fluctuations on the environment may help the population to follow the {environmental}  shift in a better way.
          \end{abstract}

          \section{Introduction}\label{intro}
\subsection{Model and motivations.}\label{subs:model_motiv}
The goal of this article is to study the evolutionary dynamics of a phenotypically structured population in an environment which varies with a linear trend but in an oscillatory manner. We study  the following non-local parabolic equation
\begin{equation}
\label{GeneralCase_ntilde}
\left\{
\begin{array}{l}
\partial_t\widetilde{n}-\sigma\partial_{xx}\widetilde{n}=\widetilde{n}[a({e(t)},x-\widetilde{c}t)-\widetilde{\rho}(t)],\quad (t,x)\in[0,+\infty)\times\R,\\
\widetilde{\rho}(t)=\dis\int_{\R}\widetilde{n}(t,x)dx,\\
\widetilde{n}(t=0,x)=\widetilde{n}_0(x).
\end{array}
\right.
\end{equation}
This equation models the dynamics of a population which is  structured by a phenotypic trait $x\in\R$. Here, $\widetilde{n}$ corresponds to the density of individuals with trait $x$. We denote by $a({e(t)},x-\widetilde{c}t)$ the intrinsic growth rate of an individual with trait $x$ at time $t$. The term $-\widetilde{c}t$ has been introduced to consider {a shifting of the fitness landscape} with a linear trend. {The function $e(t):\R_+\to E$, represents the environmental state at time $t$ (for instance, the temperature) and  is assumed to be periodic, with   $E$ corresponding to the set of the states of the environment, for instance an interval corresponding to the possible temperatures. The variation of the environmental state $e$ may have an impact on the optimal trait (the trait that maximizes $a(e,x)$) or other parameters of  selection as for instance the pressure of   selection (corresponding to the curvature of $a(e,x)$ around its maximum point, see Section \ref{Bio_example} for some examples).} The term $\widetilde{\rho}$ which corresponds to the total size of the population represents a competition term.  Here, we assume indeed a uniform competition between all the individuals. The diffusion term models the mutations, with $\sigma$ the mutation rate. 
 
 {A natural motivation to study such type of problem is the fact that many natural populations are subject both to a directional change of the phenotypic optimum and  fluctuations of the environment (\cite{chevin-et-al}). Such fluctuations may be periodic, due for instance to seasonal effects, or stochastic due to the random change of the environment.
 Here we consider a deterministic growth rate that varies with a linear trend but in an oscillatory manner. Our study provides also insights for the case of random environments since it is based on some homogenization techniques that could also be used to study the random case. However, the study of the random fluctuations would still need considerable work and is out of the scope of this article. }
 
  Will the population be able to adapt to the environmental change? Is there a {maximal} speed above which the population will get extinct? How is such maximal speed  modified due to the fluctuations?

 \subsection{Related works.}\label{subs:related_works}
 {The impact of changing environments on the evolution of quantitative traits  has been  studied using closely related quantitative genetics models in the biological literature (see for instance \cite{ML.WG.AW:91,ML.RL:93,RB.LM:95,RL.SS:96,MK.SM:14}). In these works one usually assumes that the growth rate $a$ has a particular form of quadratic type and that the environmental change has only an impact on the optimal trait. However, the environmental variations may also modify other parameters of selection, as for instance the  pressure of   selection \cite{JG.BT.HD:20}. Finally, the works studying a periodic environment, consider only a particular sinusoidal form of periodic variation \cite{ML.WG.AW:91,RL.SS:96}.}

\noindent
Models closely related to \eqref{GeneralCase_ntilde}, but with a local reaction term and no fluctuation, have been widely studied  (see for instance \cite{Berestycki-et-al, Berestycki_Rossi1, Berestycki_Rossi2,Berestycki_Fang}). Such models are introduced to study dynamics of populations structured by a space variable neglecting evolution. It is shown in particular that there exists a critical speed of  environment  change $c^*$, such that the population  survives if and only if the environment change occurs with a speed less than $c^*$.  We also refer to \cite{Bouhours2016} where an integro-difference model has been studied for the spatial dynamics of a population in the case of  a randomly changing  environment. Moreover, in  \cite{Raoul_Berest_Alfaro}, both spatial and evolutionary dynamics of a population in an environment with linearly moving optimum has been studied. While in the present work, we don't include any spatial structure, we take into account oscillatory change of environment in addition to a change with linear trend.
 \\
\noindent
The evolutionary dynamics of structured populations under periodic fluctuations of the environment has been recently studied  by \cite{Souganidis, Lorenzi-Desvillettes, Mirrahimi&Figueroa_pprint, Almeida}. The works in \cite{Lorenzi-Desvillettes, Almeida} are   focused on the study of a particular form of growth rate $a$ and in particular some semi-explicit solutions to such equations are provided. In \cite{Souganidis,Mirrahimi&Figueroa_pprint} some asymptotic analysis of such equations for general growth rates are provided. The present article is closely related to \cite{Mirrahimi&Figueroa_pprint} where a periodically evolving environment was considered without the linear trend. The presence of such linear trend of environment change leads to new difficulties in the asymptotic analysis. Moreover, we go further than the results in \cite{Mirrahimi&Figueroa_pprint} and provide an asymptotic expansion for the average size of the population in terms of the mutation rate. Such expansion is closely related to an asymptotic expansion of the Floquet eigenvalue for the linear problem. Furthermore in a very recent work \cite{Carrere_Nadin_pprint} the authors study a closely related model, but without the linear change of the environment, and study the impact of the different parameters of the model on the final population size.\\

In this article, we use an asymptotic approach based on Hamilton-Jacobi equations with constraint. This approach has been developed during the last decade to study the asymptotic solutions of selection-mutation equations, assuming small effect of the mutations. Such equations have the property that their solution concentrate as Dirac masses on the fittest traits.  There is a large literature on this approach. We refer to  \cite{DJMP,Perthame_Barles,MAA} for the establishment of the basis of this approach for homogeneous environments.

\subsection{Mathematical assumptions}\label{subs:math_assumptions}
To introduce our assumptions, we first define 
$$
\overline{a}(y)=\dis\frac{1}{T}\int_0^Ta({e(t)},y)dt.
$$
{We then assume that $e:\R_+\to E$, a  periodic function, and  $a\in L^\infty(E,C^3(\R))$ are   such that:
\begin{equation}\tag{H1}
\label{a_W3inf}
e(t)=e(t+T),\ \forall\ t\in\R_+,\quad\text{and}\ \exists\ d_0>0: \Vert a(e,\cdot)\Vert_{L^{\infty}(\R)}\leq d_0\quad\forall\ e\in E,
\end{equation}}
and that the averaged function $\overline{a}$ attains its maximum and 
\begin{equation}\tag{H2a}
\label{max_a}
\max_{x\in\R} \ \overline{a}(x)>0,
\end{equation}
which means that there exist at least some traits with strictly positive average growth rate.\\
Moreover, for some of our results (Theorem \ref{Unique_Identif} and Theorem \ref{Asymp_Expansion}) we assume that this maximum is attained at a single point $x_m$; that is
\begin{equation}\tag{H2b}
\label{x_m}
\exists! \ x_m: \max_{x\in\R} \overline{a}(x)=\overline{a}(x_m),
\end{equation}
and also 
\begin{equation}\tag{H3}
\label{x_var}
\exists! \  \overline{x}\leq x_m; \ \overline{a}(\overline{x})+\frac{\widetilde{c}^2}{4\sigma}=\overline{a}(x_m).
\end{equation}
{Let us explain the role of the trait $\overline x$ in our results. We will show in Theorem \ref{Unique_Identif} that, as the mutation rate $\sigma$ vanishes and when the speed of the environmental change $\widetilde c$ is not too high, the phenotypic density $\widetilde n$ concentrates around a single trait which moves linearly with  the same speed $\widetilde c$. The population density concentrates indeed around the trait $\overline x+\widetilde c t$ and follows in this way the optimum of the average environment, that is $x_m+\widetilde c t$, with a constant lag.\\}
Finally, we make the following assumption on the initial data:
\begin{equation}\tag{H4}
\label{n0_exp}
0\leq {\widetilde{n}_0(x)}\leq e^{C_1-C_2|x|}, \quad \forall x\in\R,
\end{equation}
which indicates that the initial density of individuals with large traits  is exponentially small.\\

\subsection{Preliminary results.}\label{subs:pre_rsults}
To avoid the shift in the growth rate $a$, we transform our problem with a change of variables. We introduce indeed $n(t,x)=\widetilde{n}(t,x+\widetilde{c}t)$ which satisfies:
\begin{equation}
\label{Case_n}
\left\{
\begin{array}{l}
\partial_tn-\widetilde{c}\partial_xn-\sigma\partial_{xx}n=n[a({e(t)},x)-\rho(t)],\quad (t,x)\in[0,+\infty)\times\R,\\
\rho(t)=\dis\int_{\R}n(t,x)dx,\\
n(t=0,x)=\widetilde{n}_0(x).
\end{array}
\right.
\end{equation}

\noindent
Next, we introduce the linearized problem associated to \eqref{Case_n}. Let $m(t,x)=n(t,x)e^{\int_0^t\rho(s)ds}$, for $n$ the solution of \eqref{Case_n}, then $m$ satisfies
\begin{equation}
\label{modelLap_in_m}
\left\{\begin{array}{cr}
\partial_t m-\widetilde{c}\partial_x m-\sigma\partial_{xx}m=a({e(t)},x)m,& (t,x)\in[0,+\infty)\times\R,\\
m(t=0,x)=\widetilde{n}_0(x), & x\in\R.\\
\end{array}
\right.
\end{equation}
We also introduce the corresponding parabolic eigenvalue problem as follows
\begin{equation}
\label{Model_EigenvalueR}
\left\{\begin{array}{cr}
\partial_tp_c-\widetilde{c}\partial_xp_c-\sigma\partial_{xx}p_c-a({e(t)},x)p_c=\lambda_{\widetilde{c},\sigma} p_c,& (t,x)\in[0,+\infty)\times\R,\\
0<p_c; \ p_c(t,x)=p_c(t+T,x), & (t,x)\in[0,+\infty)\times\R.\\
\end{array}
\right.
\end{equation}
For better legibility, we omit the tilde in the index of $p_c$, while we still refer to the problem with constant $\widetilde{c}$.
We also define the eigenvalue problem in the bounded domain $[-R,R]$, for some $R>0$,
\begin{equation}
\label{Model_EigenvalueBR}
\left\{\begin{array}{cr}
\partial_tp_R-\widetilde{c}\partial_xp_R-\sigma\partial_{xx}p_R-a({e(t)},x)p_R=\lambda_R p_R,& (t,x)\in[0,+\infty)\times [-R,R],\\
p_R=0, & (t,x)\in[0,+\infty)\times\{-R,R\},\\
0<p_R;\ p_R(t,x)=p_R(t+T,x), & (t,x)\in[0,+\infty)\times[-R,R].
\end{array}
\right.
\end{equation}
It is known that problem \eqref{Model_EigenvalueBR} has a unique eigenpair $(\lambda_R,p_R)$ with $p_R$ a strictly positive eigenfunction such that $\Vert p_R(0,\cdot)\Vert_{L^\infty([-R,R])}=1$, (see \cite{hess}). Another fundamental result (see for instance \cite{huska08}), for our purpose is that the function $R\mapsto\lambda_R$ is decreasing and $\lambda_R\rightarrow\lambda_{\widetilde{c},\sigma}$ as $R\rightarrow+\infty$.\\
To announce our first result we introduce another assumption. We assume that $a$ takes small values at infinity in the following sense: there exist positive constants $\delta$ and $R_0$ such that
\begin{equation}\tag {Hc}
\label{a_lambda_neg}
 a({e},x)+\lambda_{\widetilde{c},\sigma}\leq-\delta, \quad \forall {e\in E} \text{ and }|x|\geq R_0.
\end{equation}

\begin{proposition}
\label{Prop1}
Assume \eqref{a_W3inf}, \eqref{n0_exp} and \eqref{a_lambda_neg}. Then for problem \eqref{Model_EigenvalueR} there exists a unique generalized principal eigenfunction $p_c$ associated to $\lambda_{\widetilde{c},\sigma}$, with $\Vert p_c(0,\cdot)\Vert_{L^\infty(\R)}=1$. Moreover, we have $p_c=\dis\lim_{R\rightarrow\infty}p_R$ and
\begin{equation}
\label{bound_p}
p_c(t,x)\leq \Vert p_c\Vert_{L^\infty} e^{-\nu(|x|-R_0)},\qquad\forall(t,x)\in[0,+\infty)\times\R,
\end{equation}
for $\nu= -\frac{\widetilde{c}}{2\sg}+\sqrt{\frac{\delta}{\sg}+\frac{1}{2}\left(\frac{\widetilde{c}}{\sg}\right)^2}$.\\
Finally,  the eigenfunction $p_c(t,x)$ is exponentially stable, in the following sense; there exists $\alpha>0$ such that:
\begin{equation}
\label{Exp_Conv_m}
\Vert m(t,x)e^{t\lambda_{\widetilde{c},\sigma}}-\alpha p_c(t,x)\Vert_{L^{\infty}(\R)}\rightarrow0\quad\text{exponentially fast as}\ t\rightarrow\infty.
\end{equation}
\end{proposition}
The proof of this proposition is based on the results in \cite{Mirrahimi&Figueroa_pprint}.\\
We next define the $T-$periodic functions $Q_c(t)$ and $P_c(t,x)$ as follows:
\begin{equation}
\label{PyQ}
Q_c(t)=\dfrac{\int_{\R}a({e(t)},x)p_c(t,x)dx}{\int_{\R}p_c(t,x)dx},\quad P_c(t,x)=\dfrac{p_c(t,x)}{\int_{\R}p_c(t,x)dx},
\end{equation}
and we recall a result proved in \cite{Mirrahimi&Figueroa_pprint}.
\begin{proposition}
\label{prop2}
There exists a unique periodic solution $\widehat{\rho}(t)$ to the problem
\begin{equation}
\label{Rho_equation}
\left\{\begin{array}{l}
\dfrac{d\widehat{\rho}}{dt}=\widehat{\rho}\left[Q_c(t)-\widehat{\rho}\right],\quad t\in(0,T),\\
\widehat{\rho}(0)=\widehat{\rho}(T),
\end{array}\right.
\end{equation}
if and only if $\dis\int_0^TQ_c(t)dt>0$. Moreover this solution can be explicitly expressed as follows:
\begin{equation}
\label{Exp_rho_var}
\widehat{\rho}(t)=\frac{1-\exp\left[-\dis\int_0^{T}Q_c(s)ds\right]}{\exp\left[-\dis\int_0^{T}Q_c(s)ds\right]\dis\int_t^{t+T}\exp\left[\dis\int_t^s Q_c(\theta)d\theta\right]ds}.
\end{equation}
\end{proposition}

\subsection{The main results and the plan of the paper.}\label{subs:main_results}
{We are interested in} determining conditions on the environment shift speed $\widetilde{c}$ which leads to extinction or survival of the population. In the case of the population survival we then try to characterize asymptotically the population density considering small effect of the mutations. \\

\noindent
To present our result on the survival criterion,  we define the "critical speed".
\begin{definition}
We define the critical speed $\widetilde{c}^*_\sigma$ as follows
\begin{equation}
\label{Critical_c}
\widetilde{c}^*_\sigma=\left\{
\begin{array}{lr}
2\sqrt{-\sg\lambda_{0,\sigma}},&\text{if}\ \lambda_{0,\sigma}<0,\\
0,& \text{otherwise},
\end{array}
\right.
\end{equation}
where $\lambda_{0,\sigma}$ corresponds to the principal eigenvalue introduced by Proposition \ref{Prop1}, in the case $c=0$.
\end{definition}

The next result shows that $\widetilde{c}^*_\sigma$ is indeed a critical speed of {environmental} change above which the population goes extinct. \\
\begin{proposition}{{(Long time behavior)}}\\
\label{ConvLambda}
Let $n(t,x)$ be the solution of \eqref{Case_n}. Assume \eqref{a_W3inf}, \eqref{max_a}, \eqref{n0_exp}  and \eqref{a_lambda_neg}. Then the following statements hold:
\begin{itemize}
\item[(i)] if $\widetilde{c}\geq \widetilde{c}^*_\sigma$, then the population will go extinct, i.e.
$\rho(t)\rightarrow0,$ as $t\rightarrow\infty$, 
\item[(ii)] if $\widetilde{c}<\widetilde{c}^*_\sigma$,  then $\vert\rho(t)-\widehat{\rho}(t)\vert\rightarrow0$, as $t\rightarrow\infty$, with $\widehat{\rho}$ the unique solution of \eqref{Rho_equation}.
\item[(iii)] Moreover, $\quad\left\Vert\dfrac{n(t,x)}{\rho(t)}-P_c(t,x)\right\Vert_{L^\infty}\longrightarrow0$, as $t\rightarrow\infty$. Consequently we have, as $t\rightarrow\infty$:
\begin{equation}
\label{convPerio}
\Vert n(t,\cdot)-\widehat{\rho}(t)P_c(t,\cdot)\Vert_{L^\infty}\rightarrow 0,\:\mathrm{if}\:\widetilde{c}<\widetilde{c}^*_\sigma\quad\mathrm{and}\quad\Vert n\Vert_{L^\infty}\rightarrow0,\:\mathrm{if}\:\widetilde{c}\geq \widetilde{c}^*_\sigma.
\end{equation}
\end{itemize} 
\end{proposition}

\begin{remark}
Note that if $\lambda_{0,\sigma} \geq 0$, then $\widetilde{c}^*_{\sigma} = 0$, which means that the population goes extinct even without {environmental} linear  change, that is $\widetilde{c} = 0$.
\end{remark}
Proposition \ref{ConvLambda} allows to relate extinction/survival of the population to the {environmental} change speed and shows that if the change goes "too fast" the population will not be able to follow the environment change and will get extinct. However, if the change speed is "moderate" the phenotypic density $n$ converges to the periodic function $n_c(t,x)=\widehat{\rho}(t)P_c(t,x)$, which is in fact the unique periodic solution of \eqref{Case_n}.\\

\noindent
Next, we are interested in describing this periodic solution $n_c$, asymptotically as the effect of mutations is small. To this end, with
a change of notation, we take $\sigma = \e^2$ and $\widetilde{c}=\e c$, and we study asymptotically the solution $(n_{\e c},\widehat{\rho}_{\e c})$ as $\e$ vanishes. {For better legibility, we also define $c^*_{\e}:=\dfrac{\widetilde{c}^*_{\e^2}}{\e}$ where $\widetilde{c}^*_{\e^2}$ stands for the critical speed $\widetilde{c}^*_\sigma$ with $\sigma=\e^2$.} Note that, in view of Proposition \ref{ConvLambda}, to provide an asymptotic analysis considering $\sigma=\e^2$ small, a rescaling of the {environmental} shift speed as $\widetilde{c}=\e c$ is necessary {(see also Theorem \ref{Asymp_Expansion}). The population can tolerate only {an environmental} shift with small speed if the mutations have small effect.}\\
 In order to keep the notation simpler we denote $(n_{\e c},\widehat{\rho}_{c\e })=(n_{\e},\rho_{\e})$, which  is the unique periodic solution of the problem:
\begin{equation}
\label{n_eps}
\left\{
\begin{array}{l}
\partial_tn_\e-\e c\partial_xn_\e-\e^2\partial_{xx}n_\e=n_\e[a({e(t)},x)-\rho_\e(t)],\quad (t,x)\in[0,+\infty)\times\R,\\
\rho_\e(t)=\dis\int_\R n_\e(t,x)dx,\\
n_\e(0,x)=n_\e(T,x).
\end{array}
\right.
\end{equation}
To study asymptotically this problem we perform a Hopf-Cole transformation (or WKB ansatz), i.e we consider 
\begin{equation}
\label{HopfCole}
n_\varepsilon=\frac{1}{\sqrt{2\pi\varepsilon}}\exp{\left(\frac{\psi_\varepsilon}{\varepsilon}\right)}.
\end{equation}
This change of variable comes from the fact that with such rescaling the solution $n_\e$ will naturally have this form. While we expect that $n_\e$ tends to a Dirac mass, as $\e\to0$, $\psi_\e$ will have a non singular limit. From this transformation \eqref{HopfCole} we deduce that $\psi_\e$ solves:
\begin{equation}
\label{Eq_psi_eps}
\frac{1}{\e}\p_t\psi_\e-\e\p_{xx}\psi_\e=\left\vert\p_x\psi_\e+\frac{c}{2}\right\vert^2+a(\SM{e(t)},x)-\frac{c^2}{4}-\rho_\e(t),\quad(t,x)\in[0,+\infty)\times\R.
\end{equation}
Here is our first main result:

\begin{theorem}{{(Asymptotic behavior)}}\\
\label{Conv_HJ}
Assume \eqref{a_W3inf}, \eqref{max_a} and \eqref{a_lambda_neg} and also that {$c<\dis\liminf_{\e\to0} c^*_{\e}$}. Then the following statements hold:
\begin{itemize}
\item[(i)] As $\e\rightarrow0$, we have $\Vert\rho_\e(t)-\widetilde{\varrho}(t)\Vert_{L^\infty}\rightarrow0$,
with $\widetilde{\varrho}(t)$ {a $T-$periodic function. }

\item[(ii)] Moreover, as $\e\rightarrow0$, $\psi_\e(t,x)$ converges locally uniformly to a function $\psi(x)\in C(\R)$, a viscosity solution to the following equation:
\begin{equation}
\label{LimitEq}
\left\{\begin{array}{rcl}
\dis-\left|\p_x \psi+\frac{c}{2}\right|^2&=&\overline{a}(x)-\overline{\rho}-\frac{c^2}{4},\quad x\in\R,\\
\dis\max_{x\in\R}\psi(x)&=&0,\\
-A_1\vert x\vert^2-\frac{c}{2}x-A_2&\leq \psi\leq& c_1-c_2\vert x\vert,
\end{array}
\right.
\end{equation}
with 
$$
\bar{\rho}=\int_0^T\widetilde{\varrho}(t)dt,
$$
for some positive constants $A_1,A_2, c_1$ and $c_2=-\frac{c}{2}+\sqrt{\delta+\frac{c^2}{2}}$.\\
\end{itemize}
\end{theorem}
The above theorem is closely related to Theorem 4 in \cite{Mirrahimi&Figueroa_pprint}. A new difficulty comes from the drift term. To deal with the drift term we use a Liouville transformation (see for instance \cite{Berestycki_Rossi1, Berestycki_Rossi2}) that allows us to transform the problem to a parabolic problem without drift.\\

\noindent
{To present our next result,} let us consider the eigenproblem \eqref{Model_EigenvalueR}  for $\sigma=\e^2$ and $\widetilde{c}=c\e$, that is:
\begin{equation}
\label{Eigenvalue_lambda_ceps}
\left\{\begin{array}{cr}
\partial_tp_{c\e}-\e c\partial_xp_{c\e}-\e^2\partial_{xx}p_{c\e}-a({e(t)},x)p_{c\e}=p_{c\e}\lambda_{c,\e} ,& (t,x)\in[0,+\infty)\times\R,\\
0<p_{c\e}; \ p_{c\e}(t,x)=p_{c\e}(t+T,x), & (t,x)\in[0,+\infty)\times\R.%\\&
\end{array}
\right.
\end{equation}
Here we denote $\lambda_{c,\e}$ the eigenvalue $\lambda_{\widetilde{c},\sigma}$ with $\sigma=\e^2$ and $\widetilde{c}=c\e$ for better legibility.

\begin{theorem}{{(Uniqueness and identification of the solution)}}\\
\label{Unique_Identif}
Let $\lambda_{c,\e}$ be the principal eigenvalue of problem \eqref{Eigenvalue_lambda_ceps} and assume \eqref{a_W3inf}, \eqref{x_m}, \eqref{x_var} and \eqref{a_lambda_neg}. Assume in addition that {$c<\dis\liminf_{\e\to0} c^*_{\e}$}, then the following statements hold:
\begin{itemize} 
\item[(i)]{Let $\overline{\rho}_{\e}=\frac{1}{T}\int_0^T\rho_{\e}(t)dt$,  then $\overline{\rho}_{\e}=-\lambda_{c,\e}$.}
\item[(ii)] The viscosity solution of \eqref{LimitEq} is unique and it is indeed a classical solution given by
\begin{equation}
\label{Exp_Sol_psi}
\psi(x)=\frac{c}{2}(\overline{x}-x)+ \int_{\overline{x}}^{x_m}\sqrt{\overline{a}(x_m)-\overline{a}(y)}dy-\left\vert \int_{x_m}^{x}\sqrt{\overline{a}(x_m)-\overline{a}(y)}dy\right\vert.
\end{equation} 
where $\overline{x}<x_m$ is given in \eqref{x_var}. {Moreover, as $\e\to 0$, $\overline \rho_\e$ converges to $\overline \rho=\overline a(\overline x)$.}
\item[(iii)] Furthermore, let $n_\varepsilon$ solve \eqref{n_eps}, then
\begin{equation}
\label{Dirac_mass_conv}
n_\e(t,x)-\widetilde{\varrho}(t)\delta(x-\overline{x})\rightharpoonup 0,\quad\text{as } \e\to0,
\end{equation}
point wise in time, weakly in $x$ in the sense of measures, {with $\widetilde{\varrho}$ the unique periodic solution of the following equation
\begin{equation}
\label{eq_varrho}
\left\{\begin{array}{l}
\dfrac{d\widetilde{\varrho}}{dt}=\widetilde{\varrho}\left[a({e(t)},\bar x)-\widetilde{\varrho}\right],\quad t\in(0,T),\\
\widetilde{\varrho}(0)=\widetilde{\varrho}(T).
\end{array}\right.
\end{equation}}
\end{itemize}
\end{theorem}

\begin{remark}
The statement $(iii)$ in Theorem \ref{Unique_Identif} implies, for the solution $\widetilde{n}_\e$ to the initial problem \eqref{GeneralCase_ntilde} with $\sigma=\e^2$ and $\widetilde c=c\e$, that
\begin{equation}
\label{Dirac_mass_n_tilde}
\widetilde{n}_\e(t,x)-\widetilde{\varrho}(t)\delta(x-\overline{x}-ct)\rightharpoonup 0, \quad \text{ as } \e\to 0,
\end{equation}
point wise in time, weakly in $x$ in the sense of measures. This implies that the phenotypic density of the population concentrates on a {dominant trait which follows the optimal trait with the same speed but with a constant lag $x_m-\overline x$.}
\end{remark} 
\\
{Note that while in \cite{Mirrahimi&Figueroa_pprint} the uniqueness of the viscosity solution to the corresponding Hamilton-Jacobi equation with constraint was immediate, here to prove the uniqueness of the viscosity solution more work is required. In particular, in order to prove such result the constraint is not enough and we use also the bounds on $\psi$, given in \eqref{LimitEq}. More precisely we introduce a new function $u(x)=\psi(x)+\frac{c}{2}x$ which solves
\begin{equation}
\label{Comp_Pb_u}
\left\{
\begin{array}{c}
-\vert \p_x u\vert^2=\overline{a}(x)-\overline{\rho}-\frac{c^2}{4},\quad x\in\R,\\\\
{\max_{x\in \R} u(x)-\f{c}{2}x=0,}\\\\
-A_1\vert x\vert^2-A_2\leq u(x)\leq c_1-c_2\vert x\vert+\frac{c}{2}x,
\end{array}
\right.
\end{equation}
where the constants $A_1$, $A_2$, $c_1$, $c_2$ are the same as in \eqref{LimitEq}.\\
The main idea comes from the fact that any viscosity solution to a Hamilton-Jacobi equation of type \eqref{Comp_Pb_u} but in a bounded domain $\Omega$ can be uniquely determined by its values on the boundary points of $\Omega$ and by its values at the maximum points of the RHS of the Hamilton-Jacobi equation \cite{Lions}. \\

\noindent
{Finally, in our last result we   provide an asymptotic expansion} for the Floquet eigenvalue which leads to an asymptotic expansion for the critical speed $c^*_\e$ and the average size of the population $\bar \rho_\e$.
\begin{theorem}{(Asymptotic expansions)}\\
\label{Asymp_Expansion}
Let $\lambda_{c,\e}$ be the principal eigenvalue of problem \eqref{Eigenvalue_lambda_ceps} and assume \eqref{a_W3inf}, \eqref{x_m} and \eqref{a_lambda_neg}. Assume in addition that {$c<\dis\liminf_{\e\to0} c^*_{\e}$}, then the following asymptotic expansions hold
\begin{equation}
\label{Eigenvalue_Approx}
\overline{\rho}_{\e}=-\lambda_{c,\e}=\overline{a}(x_m)-\frac{c^2}{4}-\e \sqrt{-\overline{a}_{xx}(x_m)/2}+o(\e),
\end{equation}
\begin{equation}
\label{Crit_Speed_Approx}
c^*_{\e}=2\sqrt{\overline{a}(x_m)}-\e {\sqrt{-\f{\overline{a}_{xx}(x_m)}{2\, \overline a(x_m)}}}+o(\e).
\end{equation}
\end{theorem}
Note that the expansion for the Floquet eigenvalue is indeed related to the harmonic approximation of the   {ground state energy} of the Schr\"odinger operator \cite{Helffer}. However, here we have a parabolic, non self-adjoint operator}.\\
{In Section \ref{Bio_example} we study an illuminating biological example and show thanks to the above result that the fluctuations of the environment may help the population to follow the {environmental} shift.}\\

 \noindent
The paper is organized as follows: in Section \ref{Long_time} we deal with the long time study of the problem and prove the preliminary results Proposition \ref{Prop1} and Proposition \ref{ConvLambda}.
Next in Section \ref{Asymp_Analysis} we provide an asymptotic analysis of the problem considering small effect of mutations and we prove Theorem \ref{Conv_HJ}. {In Section \ref{Uniqueness}, we obtain the uniqueness of the viscosity solution to \fer{LimitEq} and  prove Theorem \ref{Unique_Identif}.} Section \ref{Eigen_Approx} is devoted to {the approximations} of the principal eigenvalue (average size of the population) and the critical speed, given in Theorem \ref{Asymp_Expansion}. {In Section \ref{Bio_example} we study a biological example   and discuss the effect of the fluctuations on the critical speed of survival and on the phenotypic distribution of the population. Finally}, in Appendix A and B, we provide some technical results and computations. 

\section{The convergence in long time}
\label{Long_time}

In this section we provide the proofs of Proposition \ref{Prop1} and Proposition \ref{ConvLambda}. To this end, we make a change of variable which allows us to transform the problem into a parabolic equation without the drift term.\\
Let $m(t, x)$ satisfy the  linearized problem \eqref{modelLap_in_m}, we denote $\mathcal{P}_0$ and $\mathcal{P}_c$  the linear operators associated to problem \eqref{modelLap_in_m}, for $\widetilde{c}=0$ and $\widetilde{c}>0$ respectively, that is:
\begin{equation}
\label{P0&Pc}
\mathcal{P}_0\omega:=\p_t\omega-\sg \p_{xx}\omega-a(\SM{e(t)},x)\omega,\qquad
\mathcal{P}_c\omega:=\p_t\omega-\widetilde{c}\p_x\omega-\sg \p_{xx}\omega-a(\SM{e(t)},x)\omega.
\end{equation}
In Subsection \ref{Liouville}, we introduce the Liouville transformation and provide a relation between $\mathcal{P}_0$ and  $\mathcal{P}_c$ which allows us to obtain a relationship between $\widetilde{c}$ and $\lambda_{\widetilde{c},\sigma}$. Next in Subsection \eqref{Proof_Prop1} and \eqref{Proof_Prop4} we provide the proofs of Proposition \ref{Prop1} and Proposition \ref{ConvLambda} respectively.

\subsection{Liouville transformation.}
\label{Liouville}
Here, we reduce the parabolic equation \eqref{Case_n} to a parabolic problem without the drift term via a Liouville transformation (see for instance \cite{Berestycki_Rossi1, Berestycki_Rossi2} where this transformation is used for an elliptic problem).\\
Let $M(t,x)$ be given by 
\begin{equation}
\label{Liouville_transf}
M(t,x):=m(t,x)e^{\frac{\widetilde{c}}{2\sigma}x} ,
\end{equation}
for $m(t,x)$ the solution of the linearized problem \eqref{modelLap_in_m}, then $M$ satisfies:
\begin{equation}
\label{Liouville_Eq}
\p_tM-\sg \p_{xx}M=\Big[a(\SM{e(t)},x)-\frac{\widetilde{c}^2}{4\sg}\Big]M.
\end{equation} 
We denote $\mathcal{\widetilde{P}}$ the linear operator associated to the above equation, i.e.
$$
\mathcal{\widetilde{P}}\omega:=\p_t\omega-\sg \p_{xx}\omega-a_c(\SM{e(t)},x)\omega,
$$ 
where $a_c(\SM{e(t)},x)=\left[a(\SM{e(t)},x)-\frac{\widetilde{c}^2}{4\sg}\right]$.\\
We establish in the next lemma the relation between the principal eigenvalues associated to the operators $\mathcal{P}_0$, $\mathcal{P}_c$ and $\mathcal{\widetilde{P}}$. 
\begin{lemma}
\label{lambda_c}
Let $\lambda(\mathcal{P},\mathbb{D})$ denote the principal eigenvalue of the operator $\mathcal{P}$ in the domain $\mathbb{D}$, it holds
$$
\lambda_{\widetilde{c},\sigma}=\lambda\left(\mathcal{P}_c,\R_+\times\R\right)=\lambda\left(\mathcal{\widetilde{P}},\R_+\times\R\right).
$$
Moreover, let $\lambda_{0,\sigma}=\lambda(\mathcal{P}_0,\R_+\times\R)$, then
$\lambda_{\widetilde{c},\sigma}=\lambda_{0,\sigma}+\frac{\widetilde{c}^2}{4\sg}$.
\end{lemma}
\begin{proof}
The proof follows from the definition of the eigenfunction and eigenvalue and the fact that 
$$
\mathcal{\widetilde{P}}\omega=\left(\mathcal{P}_c\left(\omega e^{-\frac{\widetilde{c}}{2\sg}x}\right)\right)e^{\frac{\widetilde{c}}{2\sg}x}.
$$
\end{proof}
\subsection{Proof of Proposition \ref{Prop1}.}
\label{Proof_Prop1}
Proposition \ref{Prop1} can be proved following similar arguments as in the proof of Lemma 6 in \cite{Mirrahimi&Figueroa_pprint}. Note that the argument in \cite{Mirrahimi&Figueroa_pprint} is based on an exponential separation result for linear parabolic equations in \cite{huska08} that holds for general linear operators of the form 
$$
\omega_t=L(t,x)\omega,\quad\text{in}\ [0,+\infty)\times\R,
$$
with $L(t,x)$ being any time-dependent second-order elliptic operator in non-divergence form, i.e: 
$$
L(t,x)\omega=a_{ij}(t,x)\p_i\p_j \omega + B_i(t,x)\p_i \omega + A(t,x)\omega,
$$
where the functions $B_i, A\in L^\infty(\R_+\times\R)$ and $a_{ij}$ satisfies
 $$
 a_{ij}(t,x)\xi_i\xi_j\geq \alpha_0|\xi|^2,\quad (t,x)\in \R_+\times\R,
 $$
(see Section  9 in \cite{huska08} for more details).\\
Here, we only provide the proof of the inequality \eqref{bound_p} which is also obtained by an {adaptation} of the proof of Lemma 6 in \cite{Mirrahimi&Figueroa_pprint}.
Let $\widetilde{a}_c(\SM{e(t)},x)=a_c(\SM{e(t)},x)+\lambda_{\widetilde{c},\sigma}$ then $p_c$ is a positive periodic solution of the following equation:
\begin{equation}
\label{a_tilda}
\partial_t p_c-\widetilde{c}\p_x p_c-\sigma\p_{xx} p_c=p_c\widetilde{a}_c(\SM{e(t)},x),\quad\mathrm{in}\;\R\times \R.
\end{equation}
Note that we have defined $p_c$ in $(-\infty, 0]$ by periodic prolongation. We denote $\Vert p_c\Vert_{L^\infty(\R\times\R)}=\Gamma$ and define:
$$
\zeta(t,x)=\Gamma e^{-\delta(t-t_0)}+\Gamma e^{-\nu(|x|-R_0)},
$$
for some $\nu$ to be found later and $\delta$, $R_0$ given in \eqref{a_lambda_neg}. One can verify that 
$$
\Gamma\leq \zeta(t,x)\quad\mathrm{if}\;|x|=R_0\;\mathrm{or}\;t=t_0.
$$
Furthermore if  $|x|>R_0$ or $t>t_0$ evaluating in \eqref{a_tilda} shows:
$$
\begin{array}{rcl}
\partial_t\zeta-\widetilde{c}\p_x\zeta-\sigma\p_{xx}\zeta-\zeta\widetilde{a}_c(\SM{e(t)},x)&=&\Gamma e^{-\delta(t-t_0)}(-\delta-\widetilde{a}_c(\SM{e(t)},x))\\
&& +\Gamma e^{-\nu(|x|-R_0)}\left(\widetilde{c}\nu\frac{x}{|x|}-\sigma\nu^2-\widetilde{a}_c(\SM{e(t)},x)\right)\geq 0,
\end{array}
$$
since $\widetilde{a}_c(\SM{e(t)},x)+\frac{\widetilde{c}^2}{4\sigma}=a(\SM{e(t)},x)+\lambda_{\widetilde{c},\sigma}\leq -\delta$ thanks to assumption \eqref{a_lambda_neg} and choosing $\nu$ conveniently such that the inequality holds. Indeed, since $-1\leq\frac{x}{|x|}\leq 1$, we have:
$$
\widetilde{c}\nu\frac{x}{|x|}-\sigma\nu^2-\widetilde{a}_c(\SM{e(t)},x)\geq -\widetilde{c}\nu-\sg\nu^2+\delta+\frac{\widetilde{c}^2}{4\sg}\geq 0
$$
for
$$
\frac{-\widetilde{c}-\sqrt{4\delta\sg+2\widetilde{c}^2}}{2\sg}\leq\nu\leq \frac{-\widetilde{c}+\sqrt{4\delta\sg+2\widetilde{c}^2}}{2\sg}.
$$
Thus $\zeta$ is a supersolution of \eqref{a_tilda} on:
$$
\Lambda_0=\{(t,x)\in (t_0,\infty)\times \R\;;|x|>R_0\},
$$
which dominates $p_c$ on the parabolic boundary of $\Lambda_0$. Applying the maximum principle to $\zeta-p_c$, we obtain
$$
p_c(t,x)\leq \Gamma e^{-\delta(t-t_0)}+\Gamma e^{-\nu(|x|-R_0)},\qquad |x|\geq R_0,\;t\in(t_0,\infty).
$$
Taking the limit $t_0\rightarrow-\infty$ yields
$$
p_c(t,x)\leq \Gamma e^{-\nu(|x|-R_0)},\qquad |x|\geq R_0,\;t< +\infty,
$$
in particular, for $\nu=\dfrac{-\widetilde{c}+\sqrt{4\delta\sg+2\widetilde{c}^2}}{2\sg}$. We conclude that $p_c$ satisfies \eqref{bound_p}.
\vbox{\hrule height0.6pt\hbox{%
   \vrule height1.3ex width0.6pt\hskip0.8ex
   \vrule width0.6pt}\hrule height0.6pt
}

\subsection{Proof of Proposition \ref{ConvLambda}.}
\label{Proof_Prop4}
The proof of Proposition \ref{ConvLambda}, is closely related to the proof of Proposition 2 in \cite{Mirrahimi&Figueroa_pprint} but we need to verify two properties before applying the arguments in \cite{Mirrahimi&Figueroa_pprint}.  To this end we prove the following lemmas. The rest of the proof follows from the arguments in \cite{Mirrahimi&Figueroa_pprint}. 
\begin{lemma}
\label{Relation_lb_c}
Let $\lambda_{\widetilde{c},\sigma}$ be the principal eigenvalue of problem \eqref{Model_EigenvalueR}. Then, $
\lambda_{\widetilde{c},\sigma}<0$ if and only if $\widetilde{c}<\widetilde{c}^*_\sigma.$
\end{lemma}
\begin{proof}
Follows directly from the definition of $\widetilde{c}^*_\sigma$.
\end{proof}\\

\begin{lemma}
\label{bound_n}
Assume \eqref{a_W3inf} and \eqref{n0_exp} and let $C_3=C_2(\sigma C_2+\widetilde{c})+d_0$ then the solution $n(t,x)$ to equation \eqref{Case_n} satisfies:
$$
n(t,x)\leq \exp\left(C_1-C_2|x|+C_3t\right),\quad\forall(t,x)\in(0,+\infty)\times\R.
$$	 
\end{lemma}
\begin{proof}
We argue by a comparison principle argument. Define the function 
$$\bar n(t,x)=\exp\left(C_1-C_2|x|+C_3t\right).$$
We prove that $n\leq \bar n$. One can verify that for $C_3$ defined as in the formulation of the Lemma, we have the following inequality a.e:
$$
\begin{array}{c}
 \partial_t \bar n -\widetilde{c}\p_x \bar n-\sigma\p_{xx} \bar n-\left[a(\SM{e(t)},x)-\rho(t)\right] \bar n  =    \\
  e^{\left(C_1-C_2|x|+C_3t\right)}\left[C_3-\sigma
C_2^2+C_2\frac{cx}{|x|}-a(\SM{e(t)},x)+\rho(t)\right]
\geq  0.
\end{array}
$$
Moreover, we have for $t=0$, $n(0,x)\leq \bar{n}(0,x)$ thanks to assumption \eqref{n0_exp}. We can then apply a maximum principle to $d(t,x)=\bar n(t,x)-n(t,x)$, in the class of $L^2$ functions, and we conclude that:
$$
0 \leq d(t,x)\Rightarrow n(t,x)\leq \bar n(t,x),\quad\forall(t,x)\in(0,+\infty)\times\R.
$$
\end{proof}

\section{Regularity estimates}
\label{Asymp_Analysis} 

In this section, we prove Theorem \ref{Conv_HJ}. To this end we first provide some uniform bounds for $\rho_\e(t)$. Then, in Subsection \ref{Regularity}, we prove that $\psi$ is locally uniformly bounded, Lipschitz continuous with respect to $x$
and locally equicontinuous in time. Finally in the last subsection we conclude the proof of Theorem \ref{Conv_HJ} by letting $\e$ go to zero and describing the limits of $\psi_\e$ and $\rho_\e$. 

\subsection{Uniform bounds for $\rho_\e$.}\label{subs:unif_bounds_rhoe} We have the following result on $\rho_\e$.
\begin{proposition}
\label{bounds_rhoe}
Assume \eqref{a_W3inf}, \eqref{a_lambda_neg} and let $\widetilde{c}=\e c$ with {$c<\dis\liminf_{\e\to0} c^*_{\e}$}. Then for all $0<\e\leq \e_0$, there exist positive constants $\rho_m$ and $\rho_M$ such that:
\begin{equation}
\label{Rho_bouded}
0 < \rho_m \leq \rho_\e(t) \leq \rho_M,\quad  \forall t\geq 0.
\end{equation}
\end{proposition}
The proof of this result follows similar arguments as in \cite{Mirrahimi&Figueroa_pprint}. For the convenience of the reader, we provide this proof in Appendix A. 

\subsection{Regularity results for $\psi_\e$.}
\label{Regularity}
In this subsection we prove some regularity estimates on $\psi_\e$ which give the basis to prove the convergence of $\psi_\varepsilon$ and $\rho_\varepsilon$ as $\varepsilon\rightarrow 0$ in Subsection \ref{limit_equation}. 
We claim the following Theorem.\\

\begin{theorem}
\label{regularity_theo}
Assume \eqref{a_W3inf}, \eqref{max_a} and \eqref{a_lambda_neg}. Let $\psi_\e$ be a $T-$periodic solution to \eqref{Eq_psi_eps}. Then the following items hold:
\begin{itemize}
\item[(i)] The sequence $(\psi_\e)_\e$ is locally uniformly bounded; i.e. 
\begin{equation}
\label{Bounds_psi_eps}
-A_1\vert x\vert^2-\frac{c}{2}x-A_2\leq \psi_\e\leq c_1-c_2\vert x\vert, \quad \forall (t,x)\in\R_+\times\R,
\end{equation}
for some positive constants $A_1,A_2, c_1$ and $c_2=-\frac{c}{2}+\sqrt{\delta+\frac{c^2}{2}}$.
\item[(ii)] Moreover, the sequence $(\phi_\e=\sqrt{2c_1-\psi_\e})_\e$, is uniformly Lipschitz continuous with respect to $x$ in $(0,+\infty)\times\R$.

\item[(iii)] Also, $(\psi_\e)_\e$ is locally equicontinuous in time in $[0,T]\times \R$ and satisfies
\begin{equation}
\label{Equi_u}
|\psi_\e(t,x)-\psi_\e(s,x)|\to0\quad\mathrm{as}\:\e\to 0, \quad \forall\ 0\leq s\leq t\leq T.
\end{equation}
\end{itemize}
\end{theorem}
In the next subsections we provide the proof of the lower bound in \eqref{Bounds_psi_eps} and the uniform Lipschitz continuity of $\phi_\e$. The proof of the other properties can be obtained by an adaptation of the arguments in \cite{Mirrahimi&Figueroa_pprint}. For the convenience of the reader we provide them in Appendix A.

\subsubsection{Lower bound for $\psi_\e$.}\label{subs:lower_bound_psie}
To obtain the lower bound for $\psi_\e$ we use the bounds for $a$ in \eqref{a_W3inf} and for $\rho_\varepsilon$ in \eqref{Rho_bouded} and we obtain for $D_0=d_0+\rho_M$
$$
\partial_t n_\e-c\e\partial_x n_\e-\e^2\partial_{xx} n_\e\geq -D_0n_\e.
$$
Let $n_\e^*$ be the solution of the following Cauchy problem
$$
\left\{
\begin{array}{l}
\partial_t n_\e^*-c\e\partial_x n^*_\e-\e^2\partial_{xx} n_\e^*+D_0 n_\e^*=0,\\
n_\e^*(0,x)=n_\e^0,
\end{array}
\right.
$$
we define $N_\e^*$ analogously to \eqref{Liouville_transf} by the Liouville transformation of $n_\e^*$ as follows
$$
N_\e^*(t,x):=n_\e^*(t,x)e^{\frac{c}{2\e}x}.
$$
Then,  $N_\e^*$ solves the heat equation
$$
\left\{
\begin{array}{l}
\partial_t N_\e^*-\e^2\partial_{xx} N_\e^*+D_1N_\e^*=0,\\
N_\e^*(0,x)=n_\varepsilon^0(x)e^{\frac{c}{2\e}x},
\end{array}
\right.
$$
for $D_1=D_0+\frac{c^2}{4}$. The solution to the latter equation is given explicitly by the Heat Kernel $K$, 
$$
N_\e^*(t,x)=e^{-D_1t}\left(N_\e^*(0,y)\ast K\right)=\frac{e^{-D_1t}}{\e\sqrt{4\pi t}}\int_{\R}N_\e^*(0,y)e^{-\frac{|x-y|^2}{4t \e^2}}dy,\quad t>0.
$$
Note that $N_\e^*(0,x)$ from its definition can be written as follows
\begin{equation}
\label{N_e*_pce}
N_\e^*(0,x):=\frac{p_{c\e}(0,x)}{\int_\R p_{c\e}(0,x)dx}\rho_\e(0)e^{\frac{c}{2\e}x}.
\end{equation}
We recall that $p_{c\e}$ is uniquely determined once $\Vert p_{c\e}(0,x)\Vert_{L^\infty(\R)}=1$ is fixed. Then, one can choose $x_\e$ such that $ p_{c\e}(0,x_\e)=1$. From an elliptic-type Harnack inequality in a bounded domain we can obtain 
\begin{equation}
\label{BoundInf_p}
p_{c\e}(t_0,x_\e)\leq \sup_{y\in B(x_\e,\e)}p_{c\e}(t_0,y)\leq C p_{c\e}(t_0,x),\quad \forall(t_0,x)\in[\delta_0,2T]\times B(x_\e,\e),
\end{equation}
where $\delta_0$ is such that $0<\delta_0<T$ and $C$ is a positive constant depending on $\delta_0$ and $d_0$ (we refer to Appendix A-Proof of  upper bound, for more details on this inequality). We then use the $T-$periodicity of $p_{c\e}$ to conclude that the last inequality is satisfied {for all} $t\in[0,T]$.\\
From \eqref{bound_p}, \eqref{N_e*_pce} and \eqref{BoundInf_p} we deduce that
$$ 
\e^{-1}D_2e^{-\frac{D_3}{\e}+\frac{c}{2\e}x}\leq \rho_m \dfrac{p_\e(0,x)e^{\frac{c}{2\e}x}}{\int_{\R}p_\e(0,x)dx}\leq N^*_\e(0,x),\quad\forall x\in B(x_\e,\e),
$$
for some positive constants $D_2$ and $D_3$ depending on $\Vert p_\e\Vert_{L^\infty}$, $\rho_m$, $\delta$, and the constants of hypothesis \eqref{a_lambda_neg}. Then, for all $(t,x)\in(0,+\infty)\times\R$
$$
\begin{array}{rcl}
N_\e^*(t,x)
&\geq &\dis\frac{D_2}{\e^{2}\sqrt{4\pi t}}e^{-\frac{D_3+\e D_1 t}{\e}}\int_{B(x_\e,\e)}e^{\frac{c}{2\e}y}e^{-\frac{|x-y|^2}{4t \e^2}}dy\\
&\geq &
\dis\frac{D_2|B(x_\e,\e)|}{\e^2\sqrt{4\pi t}}\exp\left\{-\frac{|x|^2+(|x_\e|+\e)^2}{2 t\e^2}+\frac{c}{2}\left(\frac{x_\e}{\e}-1\right)-\frac{D_3+D_1t\e}{\e}\right\}.
\end{array}
$$
By the definition of $n_\e^*$ and the comparison principle we obtain that $n_\e^*(t,x)\leq n_\e(t,x)$ and hence
$$
n_\e(t,x)\geq 
\dis\frac{D_2|B(x_\e,\e)|}{\e^2\sqrt{4\pi t}}\exp\left\{-\frac{|x|^2+(|x_\e|+\e)^2}{2 t\e^2}+\frac{c}{2}\left(\frac{x_\e-x}{\e}-1\right)-\frac{D_3+D_1t\e}{\e}\right\}.
$$
This, together with the definition of $\psi_\varepsilon$, implies that
$$
\e\log\left(\frac{D_2|B(x_\e,\e)|}{\e^{2}\sqrt{4\pi t}}\right)-\frac{|x|^2+(|x_\e|+\e)^2}{2t\e}+\frac{c}{2}(x_\e-x-\e)-(D_3+D_1t\e)\leq \psi_\e(t,x),\quad\forall t\geq 0.
$$
In particular, we obtain that $\forall t\in\left[1,1+\e T\right]$.
$$
\e\log\left(\frac{D_2|B(x_\e,\e)|}{\e^{3/2}\sqrt{4\pi t}}\right)-\frac{|x|^2+(|x_\e|+\e)^2}{2t}+\frac{c}{2}(x_\e-x-\e)-(D_3+D_1t)\leq \psi_\e\left(\frac{t}{\e},x\right)$$
Note that $x_\e$ is uniformly bounded in $\e$ thanks to \eqref{bound_p}. Then we can conclude by using the periodicity of $\psi_\e$.  We obtain a quadratic lower bound for $\psi_\varepsilon$ for all $t\geq0$; that is, there exist $A_1$, $A_2\geq 0$ and $\e_0$ small enough such that for all $\e\leq\e_0$,

\begin{equation}
\label{Low_Bound_psi_e}
-A_1|x|^2-\frac{c}{2}x-A_2\leq \psi_\e(t,x),\quad\forall t\geq 0.
\end{equation}
\subsubsection{Lipschitz bounds.}\label{subsubs:lipschitz}
In this section we prove the Lipschitz bounds for $\phi_\e$. To this end we use a Bernstein type method closely related to the one used in \cite{MAA,Mirrahimi&Figueroa_pprint}. Let $\phi_\e = \sqrt{2c_1-\psi_\e}$,  for $c_1$ given by \eqref{Bounds_psi_eps}, then $\phi_\e$ satisfies
$$
\frac{1}{\e}\partial_t\phi_\e -c\partial_x\phi_\e -\e\partial_{xx}\phi_\e -\left(\frac{\e}{\phi_\e }-2\phi_\e \right)|\partial_x\phi_\e |^2=\frac{a(\SM{e(t)},x)-\rho_\e(t)}{-2\phi_\e }.
$$
Define $\Phi_\e=\p_x \phi_\e$, which is also $T-$periodic. We differentiate the above equation with respect to $x$ and multiply by $\frac{\Phi_\e }{|\Phi_\e |}$, i.e.,
$$
\begin{array}{rcl}
\frac{1}{\e}\partial_t|\Phi_\e |-c\partial_x|\Phi_\e |-\e\partial_{xx}|\Phi_\e |-2\left(\frac{\e}{\phi_\e }-2\phi_\e \right)\Phi_\e \cdot\partial_x|\Phi_\e |+\left(\frac{\e}{\phi_\e ^2}+2\right)|\Phi_\e |^3 &&\\\\
\leq 
\dfrac{\left(a(\SM{e(t)},x)-\rho_\e(t)\right)|\Phi_\e |}{2\phi_\e ^2}-\dfrac{\partial_x a\cdot \Phi_\e }{2\phi_\e |\Phi_\e |}. &&
\end{array}
$$
From \eqref{Bounds_psi_eps} we deduce that
$$
\sqrt{c_1}\leq \phi_\e \leq \sqrt{A_1|x|^2+\frac{c}{2}x+A_3},\qquad\forall\ t\geq 0,\ x\in\R,
$$
for $A_3=A_2+2c_1$. It follows that
$$
\left|2\left(\frac{\e}{\phi_\e }-2\phi_\e \right)\right|\leq A_4|x|+A_5,
$$
for some positive constants $A_4$ and $A_5$. From here, we deduce for $\vartheta$ large enough
\begin{equation}
\label{wp_eps_theta}
\frac{1}{\e}\partial_t|\Phi_\e |-c\partial_x|\Phi_\e |-\e\partial_{xx}|\Phi_\e |-\big(A_4|x|+A_5\big)\big\vert \Phi_\e \cdot\partial_x|\Phi_\e |\big\vert+2\left(|\Phi_\e |-\vartheta\right)^3\leq0.
\end{equation}
Let $T_M>2T$ and $A_6$ to be chosen later, define now, for $(t,x)\in\Big(0,\frac{T_M}{\varepsilon}\Big]\times [-R,R]$
$$
\Theta_\e(t,x)=\frac{1}{2\sqrt{t\varepsilon}}+\frac{A_6R^2}{R^2-|x|^2}+\vartheta.
$$
We next verify that $\Theta_\e$ is a strict supersolution of \eqref{wp_eps_theta} in $\Big(0,\frac{T_M}{\varepsilon}\Big]\times [-R,R]$. To this end we compute
$$
\partial_t \Theta_\e=-\frac{1}{4t\sqrt{t\e}},\quad \partial_x \Theta_\e=\frac{2A_6R^2x}{(R^2-|x|^2)^2},\quad \partial_{xx} \Theta_\e=\frac{2A_6R^2}{(R^2-|x|^2)^2}+\frac{8A_6R^2|x|^2}{(R^2-|x|^2)^3},
$$
and then replace in \eqref{wp_eps_theta} to obtain
$$
\begin{array}{l}
\frac{1}{\e}\partial_t\Theta_\e-c\partial_x\Theta_\e-\e\partial_{xx} \Theta_\e-\big(A_4|x|+A_5\big)|\Theta_\e\cdot\partial_x \Theta_\e|+2\left(\Theta_\e-\vartheta\right)^3\\\\

=-\frac{1}{4\e t\sqrt{\e t}}-\frac{2cA_6R^2x}{(R^2-|x|^2)^2}-\e\Big[\frac{2A_6R^2}{(R^2-|x|^2)^2}+\frac{8A_6R^2|x|^2}{(R^2-|x|^2)^3}\Big]\\
-\big(A_4|x|+A_5\big)\left(\frac{1}{2\sqrt{\e t}}+\frac{A_6R^2}{R^2-|x|^2}+\vartheta\right)\frac{2A_6R^2|x|}{(R^2-|x|^2)^2}+2\left(\frac{1}{2\sqrt{\e t}}+\frac{A_6R^2}{R^2-|x|^2}\right)^3\\\\

\geq -\e \Big[\frac{2A_6R^2d}{(R^2-|x|^2)^2}+\frac{8A_6R^4}{(R^2-|x|^2)^3}\Big]
-\big(A_4R+A_5\big)\left(\frac{1}{2\sqrt{\e t}}+\frac{A_6R^2}{R^2-|x|^2}+\vartheta\right)\frac{2A_6R^3}{(R^2-|x|^2)^2}\\\\

\quad+\frac{3A_6R^2}{R^2-|x|^2}\Big(\frac{1}{2t\e}+\frac{A_6R^2}{\sqrt{\e t}(R^2-|x|^2)}\Big)+\frac{2A_6R^3}{(R^2-|x|^2)^2}\left(\frac{A_6^2R^3}{R^2-|x^2|}-c\right),
\end{array}
$$
where, for the inequality, we have used that $|x|\leq R$. \\
One can verify that the RHS of the above inequality is strictly positive for $R > 1$, $\varepsilon\leq 1$, and $A_6>>\sqrt{T_M}$. Therefore, $\Theta_\e$ is a strict supersolution of \eqref{wp_eps_theta} in $\Big(0,\frac{T_M}{\e}\Big] \times [-R,R]$ and for $\varepsilon\leq 1$.\\
We next prove that
$$
|\Phi_\e (t,x)|\leq \Theta_\e(t,x)\quad\mathrm{in}\;\Big(0,\frac{T_M}{\e}\Big]\times [-R,R].
$$
To this end, we notice that $\Theta_\e(t, x)$ goes to $+\infty$ as $|x|\rightarrow R$ or as $t\rightarrow0$. Therefore, $|\Phi_\e |(t, x) - \Theta_\e(t, x)$ attains its maximum at an interior point of $\Big(0,\frac{T_M}{\varepsilon}\Big]\times [-R,R]$. We choose $t_{\rm max} \leq \frac{T_M}{\e}$ the smallest time such that the maximum of $|\Phi_\e |(t, x)-\Theta_\e(t, x)$ in the set $(0,t_{\rm max}]\times [-R,R]$ is equal to 0. If such $t_{\rm max}$ does not exist, we are done.\\\\
Let $x_{\rm max}$ be such that $|\Phi_\e |(t, x)- \Theta_\e(t, x)\leq |\Phi_\e |(t_{\rm max}, x_{\rm max})- \Theta_\e(t_{\rm max}, x_{\rm max})= 0$ for all $(t, x) \in(0, t_{\rm max})\times [-R,R]$. At such point, we have
$$
0\leq\partial_t\big(|\Phi_\e |- \Theta_\e\big)(t_{\rm max}, x_{\rm max}),\quad 0\leq -\p_{xx}\big( |\Phi_\e |- \Theta_\e\big)(t_{\rm max}, x_{\rm max}),$$
$$  |\Phi_\e |(t_{\rm max}, x_{\rm max})\p_x |\Phi_\e |(t_{\rm max}, x_{\rm max})= \Theta_\e(t_{\rm max}, x_{\rm max})\p_x \Theta_\e(t_{\rm max}, x_{\rm max}).
$$
Combining the above properties with the facts that $|\Phi_\e |$ and $\Theta_\e$ are respectively sub- and strict super-solution of \eqref{wp_eps_theta}, we obtain that
$$
(|\Phi_\e |(t_{\rm max}, x_{\rm max})-\vartheta)^3-(\Theta_\e(t_{\rm max}, x_{\rm max})-\vartheta)^3<0\Rightarrow |\Phi_\e |(t_{\rm max}, x_{\rm max})<\Theta_\e(t_{\rm max}, x_{\rm max}),
$$
which is in contradiction with the choice of $(t_{\rm max}, x_{\rm max})$. We deduce, then that
$$
|\Phi_\e (t,x)|\leq \frac{1}{2\sqrt{\e t}}+\frac{A_6R^2}{R^2-|x|^2}+\vartheta\quad\mathrm{for}\;(t,x)\in\Big(0,\frac{T_M}{\e}\Big] \times [-R,R],\;\forall\;R>1.
$$
We note that for $\e<\e_0$ small enough we have $\frac{T_M}{\e}>\frac{2T}{\e}>\frac{T}{\e}+T>\frac{T}{\e}$. Letting $R\rightarrow\infty$ we deduce that
$$
|\Phi_\e (t,x)|\leq \frac{1}{2\sqrt{\e t}}+A_6+\vartheta\leq \frac{1}{2\sqrt{T}}+A_6+\vartheta\quad\mathrm{for}\;(t,x)\in\left[\frac{T}{\e},\frac{T}{\e}+T \right] \times \R.
$$
Finally we use the periodicity of $\Phi_\e $ to extend the result for all $t\in[0,+\infty)$ and rewriting the result in terms of $\phi_\e$ we obtain for some positive constant $A_7$
\begin{equation}
\label{Lip_phi}
|\partial_x\phi_\e|\leq A_7,\quad\text{ in }[0,+\infty)\times\R.
\end{equation}

\subsection{Derivation of the Hamilton-Jacobi equation with constraint.}
\label{limit_equation}
In this section we derive the Hamilton-Jacobi equation with constraint \eqref{LimitEq} using the regularity estimates in Theorem \ref{regularity_theo}. 

\subsubsection{Convergence along subsequences of $\psi_\e$ and $\rho_\e$.}\label{subsubs:conv_psie_rhoe}
According to section \ref{Regularity}, $\{\psi_\e\}$ is locally uniformly bounded and  equicontinuous, so by the Arzela-Ascoli Theorem after extraction of a subsequence,  $\psi_\e(t,x)$ converges locally uniformly to a continuous function $\psi(t,x)$. Moreover from \eqref{Equi_u}, we obtain that $\psi$ does not depend on $t$, i.e $\psi(t,x)=\psi(x)$.\\
Furthermore, from the uniform bounds on $\rho_\e$ in \eqref{Rho_bouded} we obtain that $|\frac{d\rho_\varepsilon}{dt}|$ is also bounded. Then we apply the Arzela-Ascoli Theorem to guarantee the locally uniform convergence along subsequences of $\rho_\e(t)$, to a function $\widetilde{\varrho}(t)$ as $\e\rightarrow 0$.
\subsubsection{The Hamilton-Jacobi equation with constraint.}
\label{Eq_psi}
Here we use a perturbed test function argument (see for instance \cite{EvansHomog}), in order to prove that, $\psi(x)=\lim_{\e\to0}\psi(t,x)$ is in fact a viscosity solution of the following Hamilton-Jacobi equation. 
\begin{equation}
\label{HJ_psi}
-\left|\p_x \psi+\frac{c}{2}\right|^2=\overline{a}(x)-\overline{\rho}-\frac{c^2}{4},
\end{equation}
where $\overline{\rho}=\frac{1}{T}\int_0^T\widetilde{\varrho}(t)dt$.
We prove that $\psi$ is a viscosity sub-solution and one can use the same type of argument to prove that it is also a super-solution.\\\\
Let us define the auxiliary ``cell problem":
\begin{equation}
\label{CellPb}
\left\{\begin{array}{cr}
\partial_t \phi=a(\SM{e(t)},x)-\widetilde{\varrho}(t)-\overline{a}(x)+\overline{\rho},&(t,x)\in[0,+\infty)\times\R^d,\\ 
\phi(0,x)=0,\\
\phi:\;T-periodic.
\end{array}
\right.
\end{equation}
This equation has a unique smooth solution, that we can explicitly write:
$$
\phi(t,x)=-t(\overline{a}(x)-\overline{\rho})+\int_0^t(a(\SM{e(t)},x)-\widetilde{\varrho}(t))dt.
$$
\\
Let $\varphi\in C^\infty(\R)$ be a test function and assume that $\psi-\varphi$ has a strict local maximum at some point $x_0\in\R$, with $\psi(x_0)=\varphi(x_0)$. We must prove:
\begin{equation}
\label{CondViscSub}
-\left|\p_x \varphi(x_0)+\frac{c}{2}\right|^2-\overline{a}(x_0)+\frac{c^2}{4}+\overline{\rho}\leq 0.
\end{equation}
We define the perturbed test function $\Psi_\e(t,x)=\varphi(x)+\e \phi(t,x)$, such that $\psi_\e-\Psi_\e$ attains a local maximum at some point $(t_\e,x_\e)$. We note that $\Psi_\e$ converges locally uniformly to $\varphi$ as $\e\rightarrow0$ since $\phi$ is locally bounded by definition, and hence one can choose $x_\e$ such that $x_\e\to x_0$ as $\e\to 0$, (see Lemma 2.2 in \cite{GuyBarles}). Then $\Psi_\e$ satisfies:
$$
\frac{1}{\e}\partial_t \Psi_\e(t_\e,x_\e)-\e\p_{xx} \Psi_\e(t_\e,x_\e)-\left|\p_x \Psi_\e(t_\e,x_\e)+\frac{c}{2}\right|^2-a(\SM{e(t_\e)},x_\e)+\frac{c^2}{4}+\rho_\e(t_\e)\leq 0,
$$
since $\psi_\e$ is a solution of \eqref{Eq_psi_eps}. The above line gives:
$$
\begin{array}{c}
\partial_t \phi(t_\e,x_\e)-\e\p_{xx} \varphi(x_\e)-\e^2\p_{xx} \phi(t_\e,x_\e)-\left|\p_x \varphi(x_\e)+\e\p_x \phi(t_\e,x_\e)+\frac{c}{2}\right|^2
\\
-a(\SM{e(t_\e)},x_\e)+\frac{c^2}{4}+\rho_\e(t_\e)\leq 0.
\end{array}
$$
Using \eqref{CellPb}, this last equation becomes:
\begin{equation}
\label{EqLim}
\begin{array}{c}
-\e\p_{xx} \varphi(x_\e)-\e^2\p_{xx} \phi(t_\e,x_\e)-\left|\p_x \varphi(x_\e)+\e\p_x \phi(t_\e,x_\e)+\frac{c}{2}\right|^2
\\
+(\rho_\e-\widetilde{\varrho})(t_\e)-\overline{a}(x_\e)+\overline{\rho}+\frac{c^2}{4}\leq 0.
\end{array}
\end{equation}
Next we pass to the limit as $\e\rightarrow0$. We know from Subsection 3.3.1 that $\rho_\e\rightarrow\widetilde{\varrho}$ locally uniformly as $\e\rightarrow0$. Moreover $\phi$ is smooth with locally bounded derivatives with respect to $x$, thanks to its definition. Using these arguments and letting $\e\rightarrow0$ in \eqref{EqLim} we obtain \eqref{CondViscSub} which implies that $\psi$ is a viscosity sub-solution of \eqref{HJ_psi}.\\\\
Furthermore, note that $\psi$ is also bounded from above, by taking the limit as $\e\rightarrow0$ in \eqref{Bounds_psi_eps}, i.e., 
\begin{equation}
\label{Bouds_psi}
\psi(x)\leq c_1-c_2|x|,
\end{equation}
and attains its maximum. We claim that
$$
\max_{x\in\R}\psi(x)=0.
$$
Indeed, from the upper bound for $\rho_\e$ in \eqref{Rho_bouded}, the definition of $\psi_\e$ in \eqref{HopfCole} and the continuity of $\psi$, we obtain that $\psi(x)\leq0$. Moreover, from the locally uniform convergence of $\psi_\e$ to $\psi$, as $\e\to0$, and \eqref{Bounds_psi_eps} we deduce that  $\max_{x\in\R} \psi(x)<0$ implies that $\psi_\e(x)<-\beta$, for all $x\in\R$ and $\e\leq\e_0$ and some positive constant $\beta$. This is in contradiction with the fact that $\rho_\e$ is bounded by below by a positive constant $\rho_m$ (we refer to section 4.3 of \cite{Mirrahimi&Figueroa_pprint} for more details). \\

%An evaluation of \eqref{HJ_psi} at a maximum point of $\psi$ gives
%\begin{equation}
%\label{rho_var}
%\overline{\rho}=\overline{a}(x^*),\quad\ x^*\in X^*.
%\end{equation}
%Note that  $\psi$ is differentiable at maximum points (with $\p_x\psi(x^*)=0$) since $\psi$, the solution to \eqref{HJ_psi}, is semi-convex (see for instance \cite{Perthame_Barles}). We conclude that $\psi$ is a viscosity solution of the Hamilton-Jacobi equation with constraint \eqref{LimitEq}. This concludes the proof of Theorem \ref{Conv_HJ}.

\section{Uniqueness of the viscosity solution to (\ref{LimitEq}) and explicit identification}
\label{Uniqueness}

In this section we {provide the proof of Theorem \ref{Unique_Identif}.}   To this end, we first derive an equivalent Hamilton-Jacobi equation to \eqref{LimitEq} by mean of the Liouville transformation  and prove some properties of the eigenvalue $\lambda_{c,\e}$. We then prove the uniqueness of the viscosity solution to such equivalent equation. This allows us to establish the uniqueness of the solution to \eqref{LimitEq} and to identify it explicitly. {Finally, we provide the proof of the convergence of $n_\e$ to the Dirac mass given by \fer{Dirac_mass_conv}.}

\subsection{Derivation of an equivalent Hamilton-Jacobi equation.}\label{subs:equiv_HJ-eqtn}
In this subsection, we define a new function 
\begin{equation}
\label{psi_u}
u(x) := \psi(x)+\frac{c}{2}x,
\end{equation}
which solves the following Hamilton-Jacobi equation in the viscosity sense
\begin{equation}
\label{HJ_u}
-|\p_xu|^2=\overline{a}(x)-\overline{\rho}-\frac{c^2}{4}.
\end{equation}
Note that the   transformation \eqref{psi_u} is indeed analogous to the Liouville transformation presented in Section   \ref{Liouville}.\\
{We} have the following boundedness result for $u$.
\begin{lemma}
\label{lem:Upp_bound_u}
{The function $u(x)$,  defined by \eqref{psi_u}, is locally bounded} and satisfies
\begin{equation}
\label{Upp_bound_u}
-A_1|x|^2-A_2\leq u(x)\leq c_1-c_2|x|+\frac{c}{2}x,\quad \forall\ x\in\R,
\end{equation}
where  the constants $A_1$, $A_2$, $c_1$ and $c_2$ are given in \eqref{Bounds_psi_eps}. 
\end{lemma}
\begin{proof}
From Subsection \ref{Regularity} we got uniform bounds for $\psi_\e$ in  \eqref{Bounds_psi_eps}, which lead to bounds on $\psi$. That is
$$
-A_1|x|^2-\frac{c}{2}x-A_2\leq\psi(x)\leq c_1-c_2|x|.
$$
Then, the bounds \eqref{Upp_bound_u} follow directly from the definition of $u(x)$ in \eqref{psi_u}. 
\end{proof}\\\\
Therefore, we conclude that the function $u$ satisfies \eqref{Comp_Pb_u}. {In Subsection \ref{sec:uniq_u} we will prove a uniqueness result for \fer{Comp_Pb_u} which will imply the uniqueness of $\psi$, the solution to \eqref{LimitEq}.}

%This Lemma helps to guarantee the uniqueness of the viscosity solution $u$ of \eqref{HJ_u}. Indeed, from the Theorem 1.11 (iii) in \cite{Lions} we can conclude the uniqueness of $u$ in a bounded domain (see the Appendix for more details), and this additional condition \eqref{Upp_bound_u} allows us to extend the uniqueness to the whole space $\R$.

\subsection{Some properties of the eigenvalue $\lambda_{c,\e}$.}
\label{ConvEigenvalue}
In this  subsection we prove Theorem \ref{Unique_Identif}-(i). {We also establish that $\lambda_{c,\e}$ is uniformly bounded above and below by negative constants and derive some properties of the limit, along subsequences, of $\lambda_{c,\e}$.} \\

\noindent
From the equation \eqref{Eigenvalue_lambda_ceps} we can integrate in $\R$, divide by $\int_\R p_{c\e}(t,x)dx$ and integrate again in $t\in[0,T]$ and obtain  
\begin{equation}
\label{Q_eps&lambda_eps}
\lambda_{c,\e}=-\dfrac{1}{T}\dis\int_0^TQ_{c\e}(t)dt.
\end{equation}
where $Q_{c\e}(t)$ is defined analogously to \eqref{PyQ} from the periodic eigenfunction $p_{c\e}$. {We next use the relationship between  the solution $n_\e$ to \eqref{n_eps} and  the eigenfunction $p_{c\e}$ to {obtain the first claim of Theorem \ref{Unique_Identif}.} Indeed, from equation \eqref{n_eps} after an integration in $x\in\R$ we obtain:
$$
\frac{d \rho_\e(t)}{d t}=\int_\R n_\e(t,x)a(\SM{e(t)},x)dx- \rho^2_\e(t).
$$
We divide by $\rho_\e(t)$ and use the relation between $n_\e$ and  $p_{c\e}$ inside of the integral, that is:
$$
\rho_\e(t)+ \frac{d}{dt}\ln\rho_\e(t)=\frac{\int_\R p_{c\e}(t,x)a(\SM{e(t)},x)dx}{\int_\R p_{c\e}(t,y)dy}.
$$
Note that the RHS is exactly $Q_{c\e}$.  We then integrate in $[0,T]$ and using \eqref{Q_eps&lambda_eps} and the $T-$periodicity of $\rho_\e$ we deduce that
\begin{equation}
\label{rho_e_lamnda_e}
\overline{\rho}_\e=-\lambda_{c,\e},
\end{equation}
}

\noindent
{
We next prove that $\lambda_{c,\e}$ is uniformly bounded above and below by negative constants. Combining \ref{Q_eps&lambda_eps} and  \eqref{a_W3inf} we obtain that
$$
-d_0\leq \lambda_{c,\e}.
$$
Moreover, since we are in the case {$c<\dis\liminf_{\e\to0} c^*_{\e}$}, we can find a positive constant $\tau$ such that for every $\e\leq\e_0$, with $\e_0$ small enough we have $c< c^*_{\e}-\tau$. Then from the definition of $c^*_{\e}$ we deduce {that}
$$
c<2\sqrt{-\lambda_{0,\e^2}}-\tau=2\sqrt{-\lambda_{c,\e}+\frac{c^2}{4}}-\tau,
$$
which leads to 
$$
\lambda_{c,\e}<-\frac{c\tau}{2}-\frac{\tau^2}{4},
$$
and hence, for $\lambda_m =\frac{c\tau}{2}+\frac{\tau^2}{4}$ we obtain
\begin{equation}
\label{lambda_m}
\lambda_{c,\e}\leq -\lambda_m<0.
\end{equation}
Thus $(\lambda_{c,\e})_\e$ is  uniformly bounded {above and below by negative constants}. This implies that we can extract a subsequence, still called $\lambda_{c,\e}$, which converges as $\e\to 0$ to some negative value $\lambda_1$. Moreover passing to the limit as $\e\to0$ in assumption \eqref{a_lambda_neg} we obtain, for all such limit values $\lambda_1$,
%\begin{equation}
%\label{a_neg_ce}
%\overline{a}(x)\leq -\delta-\lambda_{c,\e},\quad \forall|x|\geq R_0.
%\end{equation}
\begin{equation}
\label{a_neg_lambda_limit}
\overline{a}(x)\leq -\delta-\lambda_1,\quad \forall|x|\geq R_0.
\end{equation}
Note that passing to the limit of $\overline\rho_\e$ as $\e\to 0$ along the same subsequence, we obtain that
\begin{equation}
\label{rho_lambda1}
\overline{\rho}=-\lambda_1.
\end{equation}
}

\subsection{Uniqueness and explicit formula for $u(x)$.}
\label{sec:uniq_u}
In this subsection we prove the uniqueness of the viscosity solution of the Hamilton-Jacobi equation \eqref{Comp_Pb_u}. To this end we consider the Hamilton-Jacobi equation as follows
\begin{equation}
\label{MyH}
-\vert \p_x u\vert^2=h(x),\quad x\in\Omega,
\end{equation}
where $h\in C^1(\R)$. Note that this corresponds to our problem for $h(x)=\bar h(x):= \overline{a}(x)-\overline{\rho}-\frac{c^2}{4}$. \\
We divide the  proof of the uniqueness result into several steps. We first prove that, in the case where $\Omega$ is an open bounded domain and $h<0$ in $\Omega$, a viscosity solution to \eqref{MyH} can be uniquely determined {by} its values on the boundary of $\Omega$. We then use this property and \eqref{Upp_bound_u} to prove that in our problem it is not possible that $\bar h(x)<0$ for all $x\in\R$. We prove indeed that $\max \bar h(x)=0$ and this maximum is attained only at the point $x_m$. Finally we use these properties to conclude that $u$ is indeed uniquely determined by an explicit formula.

\textbf{Step 1: If $h<0$ and $\Omega$ is bounded then the viscosity solution to \eqref{MyH} is uniquely determined by its values on the boundary of $\Omega$.} Suppose that $h(x)<0$, for every $x\in\Omega$. For this problem we obtain uniqueness of the viscosity solution thanks to {a representation formula for the function $u$. Indeed, for $\Omega$ bounded we define $L(x,y)$ as follows}
\begin{equation}
\label{Def_L}
\begin{array}{rcl}
L(x,y)&=&\sup\left\{\int_0^{T_0}-\sqrt{-h(\xi(t))}dt/ (T_0,\xi)\text{ such that }\xi(0)=x,\ \xi(T_0)=y, \right. \\ 
&&\left. \xi(t)\in\Omega, \forall t\in[0,T_0], \left\vert\frac{d\xi}{dt}\right\vert\leq 1\text{ a.e in }[0,T_0]\right\},
\end{array}
\end{equation}
{and in \cite{Lions}-Chapter 5 the following is proved.}
\begin{proposition}
\label{Uniq_bounded_domain}
Assume that $h(x)<0$, $\forall x\in\Omega$, {with $\Omega$ a bounded domain}. The function 
$$
u=\dis\inf_{y\in\p\Omega}[\varphi(y)+L(x,y)],
$$
is the unique viscosity solution of 
$$
\vert D u\vert=\sqrt{-h(x)} \text{ in } \Omega;\quad u=\varphi \text{ on } \p\Omega.
$$
\end{proposition}
\textbf{Step 2: $\max_{x\in\R} \bar h(x)=\bar h(x_m)=0$ and the maximum is only attained at this point.}
We assume in the contrary that $\max_{x\in\R}\bar h(x)<0$. We consider $\Omega=B_{R'}=(-R',R')$ for $R'>0$, to be chosen later. According to step 1, we can express the value of the viscosity solution  of \eqref{MyH} at the point $0$, for $h(x)=\bar h(x)$, as follows:
{$$
u(0)=\max \left\{u(-R')-\left\vert\int_{-R'}^0\sqrt{-\bar h(y)}dy\right\vert; u(R')-\left\vert\int^{R'}_0\sqrt{-\bar h(y)}dy\right\vert\right\}.
$$}
%Then, let $R>R_0$ and choose $R'$ such that $R<<R'$ and take $x\in B_R= [-R,R]$. 
Note that thanks to \eqref{a_neg_lambda_limit} and \eqref{rho_lambda1}, we obtain that 
$$
\sqrt{\delta+\frac{c^2}{4}}\leq \sqrt{-\bar h(y)},\quad \forall \vert y\vert\geq R_0.
$$
We deduce that, for all $R'>R_0$,
$$
\int_{-R'}^0\sqrt{-\bar h(y)}dy\geq \sqrt{\delta+\frac{c^2}{4}}\,(R'-R_0),\quad
\text{ and }\quad
\int_0^{R'}\sqrt{-\bar h(y)}dy\geq \sqrt{\delta+\frac{c^2}{4}}\,(R'-R_0).
$$
Next we combine the above inequalities with the third line of \eqref{Comp_Pb_u} to obtain  
$$
\begin{array}{rl}
u(0)\leq\max &\left\{c_1-R'\sqrt{\delta+\frac{c^2}{2}}-\left(R'-R_0\right)\sqrt{\delta+\frac{c^2}{4}};\right.\ 
\\
& \left. \; \; c_1+cR'-R'\sqrt{\delta+\frac{c^2}{2}}-\left(R'-R_0\right)\sqrt{\delta+\frac{c^2}{4}}\right\},
\end{array}
$$
for $c_1$ given in \eqref{Comp_Pb_u}. This implies that, taking $R'$ arbitrarily large, $u(0)$  is arbitrarily small which is a contradiction. Therefore the assumption on $\bar h(x)$ of being strictly negative in $\Omega$ is false. \\
We have proved that $\bar h(x)$ vanishes at some point $x\in\R$. Note also from \eqref{HJ_u} that
$$
\bar{a}(x)-\bar{\rho}-\frac{c^2}{4}\leq0,
$$
and $\max_{x\in\R} \bar h(x)$ is attained at the unique maximum point of $\bar a$, which is $x_m$.\\

\textbf{Step 3: Identification of  $u$ in $\R$.} We now prove that the solution $u$ is uniquely determined by its value at the maximum point of $\bar h(x)$. That is, for all $x\in \R$\\
\begin{equation}
\label{u_x_m}
u(x)=-\left\vert\dis\int_{x_m}^x\sqrt{-\bar h(y)}dy\right\vert+u(x_m).
\end{equation}
To this end we choose $0<R$, and $0<R'$ such that $R<<R'$ and we consider the domain $\overline{B_{R'}}=[-R',x_m]\cup[x_m,R']$. Note that $\bar h<0$ in the sets $(-R',x_m)$ and $(x_m,R')$. We can thus apply Proposition \ref{Uniq_bounded_domain} in these domains:
$$
u(x)=\max \left\{u(-R')-\left\vert\int_{-R'}^{x}\sqrt{-\bar h(y)}dy\right\vert; u(x_m)-\left\vert\dis\int_{x_m}^x\sqrt{-\bar h(y)}dy\right\vert .  \right\}, \quad \forall x\in (-R,x_m),
$$
$$
u(x)=\max \left\{ u(R')-\left\vert\int_{R'}^{x}\sqrt{-\bar h(y)}dy\right\vert ; u(x_m)-\left\vert\dis\int_{x_m}^x\sqrt{-\bar h(y)}dy\right\vert .  \right\}, \quad \forall x\in (x_m,R),
$$
We next prove the following inequalities for $R'$ large enough,
$$
\left\vert\dis\int_{x_m}^x\sqrt{-\bar h(y)}dy\right\vert-u(x_m)\leq
\left\vert\int_{-R'}^x\sqrt{-\bar h(y)}dy\right\vert-u(-R'),\quad \forall x\in (-R,x_m),
$$
$$
\left\vert\dis\int_{x_m}^x\sqrt{-\bar h(y)}dy\right\vert-u(x_m)\leq 
\left\vert \int^{x}_{R'}\sqrt{-\bar h(y)}dy\right\vert-u(R'), \quad \forall x\in (x_m,R).
$$
 and combine them with the above lines to obtain \eqref{u_x_m} for all $x\in[-R,R]$. Since $R$ is arbitrary we thus obtain \eqref{u_x_m}.  \\
Suppose  that $x_m<x<R$. We prove the second inequality (the first one follows from an analogous argument). We claim that, for $R'$ large enough
\begin{equation}
\label{Sup_Ineq}
-\dis\int_{x_m}^x\sqrt{-\bar h(y)}dy+u(x_m)+\int_{x}^{R'}\sqrt{-\bar h(y)}dy-u(R')\geq 0. 
\end{equation}
Indeed, for $x\in [x_m,R]$ we have
\begin{equation}
\label{xm_x}
-\int_{x_m}^{x}\sqrt{-\bar h(y)}dy\geq -\int_{x_m}^{R}\sqrt{-\bar h(y)}dy.
\end{equation}
Moreover, from the upper bound for $u$ in \eqref{Upp_bound_u} we obtain for all $x\in B_{R'}(x_m)$,
\begin{equation}
\label{u_R'}
u(x)\leq c_1-c_2|x|+\frac{c}{2}x, \quad \Rightarrow u(R')\leq c_1+\frac{c}{2}R'.
\end{equation}
Furthermore,  following similar arguments as in the previous step we obtain that: 
\begin{equation}
\label{h_delta_c}
\sqrt{-\bar h(y)}\geq\sqrt{\delta+\frac{c^2}{4}},\quad \forall y\in(R_0,R').
\end{equation}
Finally, putting together \eqref{xm_x}, \eqref{u_R'} and \eqref{h_delta_c} we obtain:
$$
\begin{array}{l}
-\dis\int_{x_m}^x\sqrt{-\bar h(y)}dy+u(x_m)+\int_{x}^{R'}\sqrt{-\bar h(y)}dy-u(R')
\\
\geq -\dis\int_{x_m}^R\sqrt{-\bar h(y)}dy+u(x_m)+(R'-R_0)\sqrt{\delta+\frac{c^2}{4}}-\frac{c}{2}R'-c_1 
\\ \geq0,
\end{array}
$$
for $R'$ large enough.\\

\textbf{Step 4: Uniqueness of $u$.}
We finally  determine  the value of $u$ at $x_m$ which leads to the uniqueness and an explicit formula for   $u$   thanks to \fer{u_x_m}.\\
Replacing the value of $\overline h$ in \eqref{u_x_m}, we obtain
\begin{equation}
\label{Solution_u}
u(x)=u(x_m)-\left\vert \dis\int_{x_m}^{x}\sqrt{\overline{\rho}+\frac{c^2}{4}-\overline{a}(y)}dy\right\vert,\quad \forall x\in\R.
\end{equation}
This directly implies that $u$ is in fact a classical solution for $x\in\R$ which attains its maximum at $x=x_m$. We also recall the second property in \fer{Comp_Pb_u}, that is 
$$
\max_{x\in \R} u(x)-\f{c}{2}x=0.
$$
We denote the set of maximum points of $u(x)-\f{c}{2}x$ by $X^*$, i.e
$$
X^*:=\{x^* \in \R \text{ such that  $u(x^*)-\f{c}{2}x^*=0$}\}.
$$
Let $x^*\in X^*$, we evaluate the above formula of $u$ at $x^*$ in order to obtain an expression for $u(x_m)$.  This implies
\begin{equation}
\label{Solution_psi}
u(x)=\frac{c}{2} x^* +\left\vert \int_{x_m}^{x^*}\sqrt{\overline{a}(x_m)-\overline{a}(y)}dy\right\vert-\left\vert \int_{x_m}^{x}\sqrt{\overline{a}(x_m)-\overline{a}(y)}dy\right\vert,\quad \forall x\in\R.
\end{equation}
Moreover, we have
$$
u(x_m)-\f{c}{2}x_m \leq u(x^*)-\f{c}{2}x^*=0,
$$
which implies that 
$$
\frac{c}{2}(x^*-x_m)\leq 0,
$$
and hence $x^*\leq x_m$. Note also  that   we have
$\overline{a}(x_m)=\overline{\rho}+\frac{c^2}{4}$ from step 2
and  $\bar a(x^*)=\bar \rho$ thanks to \eqref{LimitEq}. Combining these properties with  assumption \eqref{x_var} it follows that $x^*=\overline{x}$, and hence $u$ is uniquely determined. As a consequence we obtain the explicit formula   \eqref{Exp_Sol_psi} for $\psi(x)$. This ends the proof of Theorem \ref{Unique_Identif}-(ii).

\subsection{Convergence to the Dirac mass.}\label{subs:conv_dirac_mass}
We deal in this subsection with the result for the convergence of $n_\varepsilon$, that is Theorem \ref{Unique_Identif}-(iii).\\
Call $f_\e(t,x)=\dfrac{n_\e(t,x)}{\rho_\e(t)}$, then $f_\e$ is uniformly bounded in $L^\infty(\R_+,L^1(\R))$. Next, we fix $t\geq0$, and we prove that $f_\e(t,\cdot)$, converges, along subsequences, to a measure, as follows
$$
f_\varepsilon(t,\cdot)\rightharpoonup\delta(\cdot-\bar{x})\quad\mathrm{as}\quad \varepsilon\rightarrow0,
$$
weakly in the sense of measures.\\
Indeed, we already know that 
$$
\max_{x\in\R}\psi(x)=\psi(\bar{x})=0\quad \text{ and }\quad \psi(x)\leq c_1-c_2\vert x\vert,
$$
for $\bar x$ given in Theorem \ref{Unique_Identif}. This implies that for any $\zeta>0$, there exists $\beta>0$ such that $\psi(x)\leq -\beta$ for every $x\in\R\setminus [\bar x-\zeta,\bar x+\zeta]$. \\
We denote $\mathcal{O}=\R\setminus [\bar{x}-\zeta,\bar{x}+\zeta]$  and choose $\chi\in C_c(\mathcal{O})$, such that $\mathrm{supp}\; \chi\subset \mathcal{K}$, for some compact set $\mathcal{K}$, then it follows that
$$
\left|\int_{\mathcal{O}}f_\e(t,x)\chi(x)dx\right|\leq \frac{1}{\rho_m}\int_{\mathcal{O}}e^{\frac{\psi_\e(t,x)}{\e}}|\chi(x)|dx\leq \frac{1}{\rho_m}\int_{\mathcal{K}}e^{{\frac{\psi_\e(t,x)}{\e}}}|\chi(x)|dx.
$$
From the locally uniform convergence of $\psi_\e$, to $ \psi(x)$, we obtain that there exists $\e_0>0$ such that $\forall\e<\e_0$,  $\psi_\e(t,x)\leq-\frac{\beta}{2}$, $\forall x\in\mathcal{K}$, and hence
$$
\int_{\mathcal{K}}e^{\frac{\psi_\e(t,x)}{\e}}|\chi(x)| dx\leq\int_{\mathcal{K}}e^{-\frac{\beta}{2\e}}|\chi(x)|dx\rightarrow0\quad\mathrm{as}\;\e\rightarrow0,
$$
since $\chi$ is bounded in $\mathcal{K}.$ Therefore, thanks to the  uniform $L^1$ bound of $f_\e$, we obtain that $f_\e$ converges weakly in the sense of measures and along subsequences to $\mu\delta(x-\bar{x})$ as $\e\to0$.  Then to prove that in fact, $\mu=1$ we can proceed as in Section 4.3 in \cite{Mirrahimi&Figueroa_pprint}. 

 \subsection{Identification of the limit of $\rho_\varepsilon$.}\label{identif_rhoe}
In order to identify the limit of $\rho_\varepsilon$ we first write thanks to Proposition \ref{ConvLambda} and Proposition \ref{prop2}  the following explicit expression for $\rho_\e$:
\begin{equation}
\label{Exp_rho_eps}
\rho_\e(t)=\frac{1-\exp\left[-\dis\int_0^{T}Q_{c\e}(s)ds\right]}{\exp\left[-\dis\int_0^{T}Q_{c\e}(s)ds\right]\dis\int_t^{t+T}\exp\left[\dis\int_t^s Q_{c\e}(\theta)d\theta\right]ds},
\end{equation}
where $Q_{c\e}$ is defined analogously to \eqref{PyQ}, using the periodic eigenfunction $p_{c\e}$ of problem \eqref{Eigenvalue_lambda_ceps}. \\
We then compute the limit of $Q_{c\e}$ as $\e\to0$. We know that $p_{c\e}(t,x)=\dfrac{n_\e(t,x)}{\rho_\e(t)}\dis\int_{\R}p_{c\e}(t,y)dy$. {Replacing $p_{c\e}$ by this quantity in the formula for $Q_{c\e}$} we obtain
$$
\begin{array}{rl}
Q_{c\e}(t)&=\dfrac{\dis\int_{\R}a(e(t),x)p_{c\e}(t,x)dx}{\dis\int_{\R}p_{c\e}(t,x)dx}
=\dfrac{\dis\int_{\R}a(\SM{e(t)},x)\dfrac{n_\e(t,x)}{\rho_\e(t)}\int_{\R}p_{c\e}(t,y)dy dx}{\dis\int_{\R}p_{c\e}(t,x)dx}\\
&=\dfrac{\dis\int_{\R}a(e(t),x)n_\e(t,x)dx}{\rho_\e(t)}.
\end{array}
$$
{From the previous subsection} we deduce that:
$$
\lim_{\e\rightarrow0}Q_{c\e}(t)=\lim_{\e\rightarrow0}\int_{\R}f_\e(t,x)a(\SM{e(t)},x)dx=a(\SM{e(t)},\bar x).
$$
Finally we can pass to the limit in the expression \eqref{Exp_rho_eps} for $\rho_\e$, to obtain the following explicit formula for $\tilde{\varrho}$
\begin{equation}
\label{Exp_rho}
\widetilde{\varrho}(t)=\frac{1-\exp\left[-\dis\int_0^{T}a(s,\bar x)ds\right]}{\exp\left[-\dis\int_0^{T}a(s,\bar x)ds\right]\dis\int_t^{t+T}\exp\left[\dis\int_{t}^s a(\theta,\bar x)d\theta\right]ds},
\end{equation} 
 which is in fact the unique periodic solution of the equation \eqref{eq_varrho} {thanks to Proposition \ref{prop2}}. Therefore using the convergence result for $\rho_\varepsilon$ we deduce finally \eqref{Dirac_mass_conv} and this ends the proof of Theorem \ref{Unique_Identif}.

\section{Approximations of the eigenvalue}
\label{Eigen_Approx}

In this section we prove Theorem \ref{Asymp_Expansion}, i.e., the asymptotic expansions \eqref{Eigenvalue_Approx} and \eqref{Crit_Speed_Approx}. Note that the first equality in \eqref{Eigenvalue_Approx} has been already obtained in Section \ref{ConvEigenvalue}. \\
To this end we develop an asymptotic approximation {of order $\e$ of the eigenvalue $\lambda_{c,\e}$ given by the eigenvalue problem \eqref{Eigenvalue_lambda_ceps}}.
To obtain such asymptotic expansion we construct an approximate eigenfunction
$\widetilde{p}_\e$ corresponding to an approximate eigenvalue $\widetilde{\lambda}_\e$ which solves an equation close to \eqref{Eigenvalue_lambda_ceps}. We then use Proposition \ref{Prop1} to prove that $\widetilde{\lambda}_\e$ approximates $\lambda_{c,\e}$ with an error of order $\e^2$. \\
To construct an approximate eigenfunction, we first try to approximate $w_\e$, obtained from the Hopf-Cole transformation of $p_{c\e}$ as follows:
\begin{equation}
\label{p_ce_e=w_e}
p_{c\e}(t,x)=\frac{1}{\sqrt{2\pi\e}}e^{\frac{w_\e(t,x)}{\e}}.
\end{equation}
One can verify that $w_\e$ solves:
\begin{equation}
\label{v_eps}
\frac{1}{\e}\partial_tw_{\e}-\e\partial_{xx}w_{\e}=\left\vert\partial_x w_\e+\frac{c}{2}\right\vert^2+a(\SM{e(t)},x)+\lambda_{c,\e}-\frac{c^2}{4}.
\end{equation}
We can obtain similar bounds for $w_\e$ as for $\psi_\e$, which guarantee the convergence along subsequences of $w_\e$ to certain function $w=w(x)$, which is in fact the limit of the whole sequence $w_\e$, and satisfies the following Hamilton-Jacobi equation in the viscosity sense
\begin{equation}
\label{HJ_w}
-\left\vert\partial_x w+\frac{c}{2}\right\vert^2=\overline{a}(x)+\lambda_1-\frac{c^2}{4}.
\end{equation}

\begin{remark}
Note that in order to obtain the limit equation \eqref{HJ_w} we can argue exactly as for $\psi_\e$ in Section \ref{Eq_psi}, by a "perturbed test function" argument, (see also \cite{Mirrahimi&Figueroa_pprint}). 
\end{remark}
{Note that, thanks to theorems \ref{Conv_HJ} and \ref{Unique_Identif},  $\lambda_1=-\overline \rho$ and   $ \psi(x)$ is a solution to \eqref{HJ_w}.}   We then, write (formally)
\begin{equation}
\label{Formal_exp_phi_lambda}
w_\e(t,x)=\psi(x)+\e \phi(t,x)+\e^2\omega(t,x)+o(\e^2)\quad\text{and} \quad \lambda_{c,\e}={-\overline \rho}+\e\lambda_2+o(\e),
\end{equation}
for some $T-$periodic functions $\phi$ and $\omega$, and we construct the following approximated eigenpair:
\begin{equation}
\label{Def_approx_eigenpair}
\widetilde{\psi_\e}=\psi+\e \phi\quad\text{and} \quad \widetilde{\lambda}_{\e}={-\overline \rho}+\e\lambda_2.
\end{equation} 
We then substitute this pair $(\widetilde{\psi_\e}, \widetilde{\lambda}_{\e})$ into \eqref{v_eps} and obtain:
\begin{equation}
\label{Eq_Approx_eigenpair}
\p_t \phi-c\psi_x-c\e\p_x\phi-\e\psi_{xx}-\e^2\p_{xx} \phi=|\psi_x +\e\p_x \phi|^2+a(\SM{e(t)},x) {-\overline \rho}+\e\lambda_2+o(\e),
\end{equation}
where the notations $\psi_x$ and $\psi_{xx}$ correspond respectively to the first and second derivative of $\psi$.\\ 
Regrouping {the terms with} similar powers of $\e$ we obtain the following system for $\phi$,
\begin{equation}
\label{System_approx}
\left\{
\begin{array}{l}
\dis\p_t \phi=\left|\psi_x +\frac{c}{2}\right|^2+a(\SM{e(t)},x)-\frac{c^2}{4} {-\overline \rho},\\\\
\dis-\psi_{xx}=[2\psi_x +c]\frac{1}{T}\int_0^T\p_x \phi(t,x)dt+\lambda_2.
\end{array}
\right.
\end{equation}
We remark that the previous system has a unique solution $\phi$ up to addition by a constant. Indeed, from equation \eqref{LimitEq} we obtain
$$ 
\dis\p_t \phi=a(\SM{e(t)},x)-\bar{a}(x).
$$
Integrating in $[0,t]$ leads to
$$
\phi(t,x)=\phi(0,x)+\int_0^ta(\tau,x)d\tau-t\bar{a}(x),
$$
and the value of $\phi(0,x)$ can be obtained from the second equation in \eqref{System_approx} once we fix $\phi(0,x_m)$.
Note that here we use the fact that $2\psi_x + c$ vanishes only at the point $x_m$. \\
We now define $\widetilde{p}_\e(t,x):=\frac{1}{\sqrt{2\pi\e}}e^{\frac{\widetilde{\psi}_\e(t,x)}{\e}}$, and use the system \eqref{System_approx}, to obtain the equality: 
\begin{equation}
\label{App_p_eps_tilde}
\mathcal{P_\e}\widetilde{p}_\e-\widetilde{\lambda}_\e\widetilde{p}_\e=-\e^2\left(|\p_x \phi|^2+\p_{xx}\phi\right)\widetilde{p}_\e,
\end{equation}
for $\mathcal{P_\e}$ the following parabolic operator
$$
\mathcal{P_\e}p=\p_tp-c\e\p_xp-\e^2\p_{xx}p-a(\SM{e(t)},x)p.
$$
We  denote
$$
\lambda_{\e^+}=\widetilde{\lambda}_\e+\e^2K, \quad\text{ and }\quad \lambda_{\e^-}=\widetilde{\lambda}_\e-\e^2K,
$$
with
\begin{equation}
\label{Def_K}
K=\Vert \vert\p_x\phi\vert^2+\p_{xx}\phi\Vert_{L^\infty},
\end{equation}
where the well definition of $K$ is guaranteed by the next lemma which is proved in the next subsection. 
\begin{lemma}
\label{Borne_operat}
The constant $K$ given in \eqref{Def_K} is well defined. Moreover the function $\phi$ computed above solves \eqref{System_approx} with 
%$\lambda_1=-\bar a(x_m)+\frac{c^2}{4}$ and
{$\lambda_2=\sqrt{-\bar a_{xx}(x_m)/2}$}.
\end{lemma}
{We then deduce from \eqref{App_p_eps_tilde} that}
$$
\widetilde{p}_\e\lambda_{\e^-}\leq \p_t \widetilde{p}_\e-c\e \p_x \widetilde{p}_\e -\e^2\p_{xx}\widetilde{p}_\e-a(\SM{e(t)},x)\widetilde{p}_\e\leq\widetilde{p}_\e\lambda_{\e^+}.
$$
{We next} define the functions 
$$
\overline{q}_\e(t,x)=\widetilde{p}_\e(t,x)e^{- t\lambda_{\e^-}},\quad \underline{q}_\e(t,x)=\widetilde{p}_\e(t,x)e^{-t \lambda_{\e^+}}.
$$
One can verify that $\overline{q}_\e$ and $\underline{q}_\e$ are super- and sub-solution of the linear problem \eqref{modelLap_in_m} with $\sigma=\e^2$ and $\widetilde{c}=c\e$, that is
$$
\partial_t \underline{q}_\e-c\e\p_x \underline{q}_\e-\e^2\p_{xx} \underline{q}_\e \leq \underline{q}_\e a(\SM{e(t)},x),
$$
$$
\partial_t \overline{q}_\e-c\e\p_x \overline{q}_\e-\e^2\p_{xx} \overline{q}_\e \geq \overline{q}_\e a(\SM{e(t)},x).
$$
We then apply a Comparison Principle and obtain that the solution $q_\e(t,x)$ to the following linear problem
\begin{equation}
\label{linear_eps}
\left\{
\begin{array}{ll}
\partial_t q_\e-c\e\p_x q_\e-\e^2\p_{xx} q_\e = q_\e a(\SM{e(t)},x),\\
q_\e(0,x)=\widetilde{p}_\e(0,x),
\end{array}
\right.
\end{equation}
satisfies
$$
\underline{q}_\e(t,x)\leq q_\e(t,x)\leq \overline{q}_\e(t,x),\quad\forall(t,x)\in\R_+\times \R.
$$
From the proof of Proposition 1 in Section 2.2 (see equation \eqref{Exp_Conv_m}), applied to the case $\sigma=\e^2$ and $c=c\e$ we know that $q_\e$
converges exponentially fast as $t\to +\infty$ to the periodic eigenfunction in \eqref{Eigenvalue_lambda_ceps},  (see also \cite{huska08}); that is, we can write for some positive constants $\alpha$ and $\beta$, 
\begin{equation}
\label{Conv_linear_eps}
\Vert q_\e e^{t\lambda_{c,\e}}-\alpha p_{c\e}\Vert_{L^\infty}\leq  e^{-\beta t}.
\end{equation}
We recall that $q_\e e^{t\lambda_{c,\e}}$ can indeed be written as 
$$
q_\e e^{t\lambda_{c,\e}}=\widetilde{q}_{\e,1}+\widetilde{q}_{\e,2},
$$
with $\widetilde{q}_{\e,1}(t,\cdot)\in \text{span}\{p_{c\e}(t,\cdot)\}$, $\widetilde{q}_{\e,2}\to 0$ exponentially fast and
$$
\int_\R \widetilde{q}_{\e,2}(t,x)p_{c\e}^*(t,x)dx=0,
$$
where $p_{c\e}^*$ is the principal eigenfunction to the adjoint problem
\begin{equation}
\label{Pb_adjoint}
-\p_t p_{c\e}^*+c\e\p_x p_{c\e}^*-\e^2\p_{xx}p_{c\e}^*=(a(\SM{e(t)},x)+\lambda_{c,\e})p_{c\e}^*,
\end{equation}
(see Theorem 2.2 {in} \cite{huska08} and the proof of Lemma 6  {in} \cite{Mirrahimi&Figueroa_pprint}). The positivity of $\alpha$ is then derived from the fact that $q_\e(0,x)$ and $p_{c\e}^*$ are positive functions. \\ 
On the one hand equation \eqref{Conv_linear_eps} implies that,
$$
0\leq {\widetilde{p}_\e} e^{(-\lambda_{\e^+}+\lambda_{c,\e})t}\leq \alpha p_{c\e}+e^{-\beta t}.
$$
Since $p_{c\e}$ and $\widetilde{p}_\e$ are time-periodic functions, then necessarily 
$$
\lambda_{c,\e}-\lambda_{\e^+}\leq 0,
$$
otherwise we get a contradiction as $t\to+\infty$. Therefore\
\begin{equation}
\label{Left_eigenvalue}
\lambda_{c,\e}-\widetilde{\lambda}_\e\leq K \e^2,
\end{equation}
where $K$ is defined in \eqref{Def_K}.\\
On the other hand, from \eqref{Conv_linear_eps} we obtain
$$
{\widetilde{p}_\e}e^{(-\lambda_{\e^-}+\lambda_{c,\e})t}\geq \alpha p_{c\e}-e^{-\beta t}.
$$
Note that if $\lambda_{c,\e}-\lambda_{\e^-}\leq 0$ we obtain from the $T-$periodicity of the eigenfunctions, as $t\to+\infty$, that  $p_{c\e}\leq 0$, which is also a contradiction. {We deduce} that
$$
\lambda_{c,\e}-\lambda_{\e^-}\geq 0.
$$
Therefore we have
\begin{equation}
\label{Right_eigenvalue}
\lambda_{c,\e}-\widetilde{\lambda}_\e\geq -K\e^2. 
\end{equation}
Combining both inequalities \eqref{Left_eigenvalue} and \eqref{Right_eigenvalue} we write
\begin{equation}
\label{Borne_eigenvalues}
\big \vert \lambda_{c,\e}-({-\overline \rho}+\e\lambda_2)\big \vert\leq K\e^2,
\end{equation}
which leads thanks to {Theorem \ref{Unique_Identif}-(ii) and} Lemma \ref{Borne_operat} to an approximation for the eigenvalue of order $\e^2$ as follows:
$$
\lambda_{c,\e}=-\overline{a}(x_m)+\frac{c^2}{4}+\e \sqrt{-\bar{a}_{xx}(x_m)/2}+o(\e).
$$
The approximation \eqref{Crit_Speed_Approx} for the critical speed $c_\e^*$ can be derived from the above approximation and \eqref{Critical_c}. Indeed, from \eqref{Critical_c} and the definition of $c_\e^*$ we obtain
$$
c_\e^*=2\sqrt{\bar a(x_m)-\e\sqrt{-\frac{\bar a_{xx}(x_m)}{2}}+o(\e)}=2\sqrt{\bar a(x_m)}-\e {\sqrt{-\frac{\bar a_{xx}(x_m)}{2\,\bar a(x_m)}}}+o(\e).
$$
\subsection{Boundedness of $K$.} In this subsection we prove Lemma \ref{Borne_operat}. We provide the proof in several steps.\\
\begin{proof}(\textbf{Proof of Lemma \ref{Borne_operat}})
\subparagraph{{Step 1:} $\vert \p_x \phi\vert$ is bounded.}
An integration in $[0,T]$ of the first equation in \eqref{System_approx} gives us the
already known equation for $\psi$ in {\eqref{HJ_psi}. This allows us to rewrite the equation} as follows:
\begin{equation}
\label{Deriv_x_phi}
\p_t\phi=a(\SM{e(t)},x)-\overline{a}(x)\Rightarrow \p_x \phi(t,x)=\p_x \phi(0,x)+\int_0^t a_x(\SM{e(\tau)},x)d\tau-t\overline{a}_x(x),
\end{equation}
where $a_x$ and $\overline{a}_x$ denote the derivatives with respect to $x$ of $a(\SM{e(t)},x)$ and $\overline{a}(x)$   respectively. This implies that in order to bound $\p_x\phi$ we just need to bound the derivative of $\phi$ at point $t=0$ since $a(\SM{e(t)},x)\in L^\infty(\R_+,C^3(\R))$.\\

Then from the second equation in \eqref{System_approx} we obtain:
\begin{equation}
\label{Deriv_phi_def}
\frac{1}{T}\int_0^T\p_x\phi(t,x)dt=\dfrac{-\psi_{xx}(x)+\psi_{xx}(x_m)}{2\psi_x(x)+c},
\end{equation}
if the last formula is well defined, i.e., if all the derivatives exist.\\
Note that, an integration for $t\in[0,T]$ in the equation \eqref{Deriv_x_phi} leads to, (after dividing by $T$)
$$
\frac{1}{T}\int_0^T\p_x \phi(t,x)dt=\p_x \phi(0,x)+\frac{1}{T}\int_0^T\int_0^t a_x(\SM{e(\tau)},x)d\tau dt-\frac{T}{2}\overline{a}_x(x),
$$
since $\p_x \phi(0,x)$ does not depend on $t$. We then deduce from the last formula and \eqref{Deriv_phi_def}
\begin{equation}
\label{Deriv_phi_0}
\p_x\phi(0,x)=\dfrac{-\psi_{xx}(x)+\psi_{xx}(x_m)}{2\psi_x(x)+c}+G(x),
\end{equation}
where 
\begin{equation}
\label{G(x)}
G(x)=-\frac{1}{T}\int_0^T\int_0^t a_x(\SM{e(\tau)},x)d\tau dt+\frac{T}{2}\overline{a}_x(x),
\end{equation}
is a regular function. We next prove that the derivatives involved in \eqref{Deriv_phi_def} exist. To this end we claim the following technical result.
\begin{lemma}
\label{Diff_psi}
The function $\psi(x)$ is twice differentiable for every $x\in\R$ and 
\begin{equation}
\label{Deriv2_psi}
\psi_{xx}(x)=\left\{
\begin{array}{lr}
-\dfrac{\overline{a}_x(x)}{2\sqrt{\overline{a}(x_m)-\overline{a}(x)}},& x<x_m,\\\\
-\sqrt{-\overline{a}_{xx}(x_m)/2},& x=x_m,\\\\
\dfrac{\overline{a}_x(x)}{2\sqrt{\overline{a}(x_m)-\overline{a}(x)}},& x>x_m.
\end{array}
\right.
\end{equation}
\end{lemma}
\begin{proof}(\textbf{Proof of Lemma \ref{Diff_psi}})\\
Indeed, from the explicit formula \eqref{Exp_Sol_psi} we differentiate and obtain:
\begin{equation}
\label{Deriv_psi}
\psi_x(x)=\left\{
\begin{array}{lr}
-\frac{c}{2}+\sqrt{\overline{a}(x_m)-\overline{a}(x)},& x<x_m,\\
-\frac{c}{2}, & x=x_m,\\
-\frac{c}{2}-\sqrt{\overline{a}(x_m)-\overline{a}(x)},& x>x_m.
\end{array}
\right.
\end{equation}
We next compute 
$$
\lim_{x\to x_m^+}\frac{\psi_x(x)-\psi_x(x_m)}{x-x_m}=\lim_{x\to x_m^+}\frac{-\sqrt{\bar{a}(x_m)-\bar a(x)}}{x-x_m}=\lim_{x\to x_m^+}\frac{-\sqrt{f(x)}}{x-x_m},
$$
where we have denoted $f(x)=\overline{a}(x_m)-\overline{a}(x)$.  We write a Taylor expansion of $f$ around $x=x_m$, i.e.:
$$
f(x)=-\frac{1}{2}\bar a_{xx}(x_m)(x-x_m)^2-\frac{1}{6}\bar a_{xxx}(x_m)(x-x_m)^3+o((x-x_m)^3),
$$
since $f(x_m)=0$ and $x_m$ is a maximum point. It implies that:
$$
\lim_{x\to x_m^+}\frac{\psi_x(x)-\psi_x(x_m)}{x-x_m}=\lim_{x\to x_m^+}\dfrac{\bar a_x(x)}{2\sqrt{\bar a(x_m)-\bar a(x)}}=-\sqrt{-\overline{a}_{xx}(x_m)/2}.
$$
Note that $x_m$ being a maximum point, $\overline{a}(x_m)\geq\overline{a}(x)$, $\forall x\in\R$ and $\overline{a}_{xx}(x_m)\leq 0$. Following similar arguments one can prove that 
$$
\lim_{x\to x_m^-}\frac{\psi_x(x)-\psi_x(x_m)}{x-x_m}=\lim_{x\to x_m^-}\dfrac{-\bar a_x(x)}{2\sqrt{\bar a(x_m)-\bar a(x)}}=-\sqrt{-\overline{a}_{xx}(x_m)/2}.
$$

\end{proof}
We pursue with the proof of Lemma \ref{Borne_operat}.\\
By substituting the derivatives of $\psi$ in \eqref{Deriv_phi_def} we obtain for every $x\neq x_m$:
{
\begin{equation}
\label{Deriv_phi}
\p_x\phi(0,x)=G(x)+\left\{
\begin{array}{lr}
\dfrac{\overline{a}_x(x)-\sqrt{-2\overline{a}_{xx}(x_m)(\overline{a}(x_m)-\overline{a}(x))}}{4(\overline{a}(x_m)-\overline{a}(x))},& x<x_m,\\\\
\dfrac{\overline{a}_x(x)+\sqrt{-2\overline{a}_{xx}(x_m)(\overline{a}(x_m)-\overline{a}(x))}}{4(\overline{a}(x_m)-\overline{a}(x))},& x>x_m.
\end{array}
\right.
\end{equation}
\\\\
We can bound $\p_x\phi(0,x)$ near to $x=x_m$. We write the limits as $x\to x_m$ in \eqref{Deriv_phi} in terms of $f$ and its derivatives $A=-\overline{a}_{xx}(x_m)/2$ and $B=-\overline{a}_{xxx}(x_m)/6$, and compute:
\begin{equation}
\label{Limit_d_phi}
\dis\lim_{x\to x_m^{\mp}}\p_x\phi(0,x)=\lim_{x\to x_m^{\mp}} G(x)+\lim_{x\to x_m^{\mp}} \frac{- f'(x)\mp 2\sqrt{A}\sqrt{f(x)}}{4 f(x)},
\end{equation}
if both limits exist. Note that from the definition of $G$ in \eqref{G(x)} {we deduce that the first limit in \eqref{Limit_d_phi} exists and is equal to $G(x_m)$}. We, then, only need to compute the second one to guarantee the existence of $\p_x\phi(0,x_m)$. We compute both lateral limits separately:
$$
\begin{array}{l}
\lim_{x\to x_m^{-}} \dfrac{- f'(x)-2\sqrt{A}\sqrt{f(x)}}{4 f(x)}\\\\
=\lim_{x\to x_m^{-}}\frac{-[2A(x-x_m) +3B( x-x_m)^2+o((x-x_m)^2)]- 2A\vert x-x_m\vert\sqrt{1+\frac{B}{A}( x-x_m)+o(( x-x_m) )}}{4\left(A(x-x_m)^2+B(x-x_m)^3+o((x-x_m)^3\right))}\\\\
=\lim_{x\to x_m^{-}}\frac{-2A(x-x_m) -3B( x-x_m)^2+o((x-x_m)^2)+ 2A(x-x_m)(1+\frac{B}{2A}( x-x_m)+o(( x-x_m) )}{4\left(A(x-x_m)^2+B(x-x_m)^3+o((x-x_m)^3\right))}\\\\
=-\dfrac{B}{2A}.
\end{array}
$$
Following similar arguments one can prove that 
$$
\dis\lim_{x\to x_m^{+}} \dfrac{- f'(x)-2\sqrt{A}\sqrt{f(x)}}{4 f(x)}=-\dfrac{B}{2A}.
$$
From this last computation and formula \eqref{Deriv_phi} we deduce that $\p_x\phi(t,x)$ is bounded for every $(t,x)\in\R_+\times\R$ and
$$
\p_x\phi(0,x_m)=G(x_m)-\frac{\overline{a}_{xxx}(x_m)}{6\overline{a}_{xx}(x_m)}.
$$  
}
\subparagraph{{Step 2:} $\vert \p_{xx} \phi\vert$ is bounded.} {Again in order to bound $\p_{xx} \phi(t,x)$ we can bound $\p_{xx} \phi(0,x)$ according to formula \eqref{Deriv_x_phi}. Note that far from $x_m$ this derivative exists and it is bounded because of the regularity of $a(\SM{e(t)},x)$. To verify the boundedness near of $x_m$ we follow the same arguments as above for the first derivative, that is, we denote $f(x)=\bar a(x_m)-\bar a(x)$ as before and we compute 
\begin{equation}
\begin{array}{rcl}
\label{limit_d2_phi}
\dis\lim_{x\to x_m^{\mp}}\frac{\p_x\phi(0,x)-\p_x\phi(0,x_m)}{x-x_m}&=&\dis\lim_{x\to x_m^{\mp}}\frac{G(x)-G(x_m)}{x-x_m} \\
&+& \dis\lim_{x\to x_m^{\mp}} \frac{- f'(x)\mp 2\sqrt{A}\sqrt{f(x)}+\frac{2B}{A}f(x)}{4 f(x)(x-x_m)},
\end{array}
\end{equation}
if both limits in the RHS exist and are bounded. {Note that the first limit in the RHS of \eqref{limit_d2_phi} exist and is equal to $G'(x_m)$ because of the definition of $G$ in \eqref{G(x)} and the regularity of $a(\SM{e(t)},x)$.} 
Moreover, using the Taylor expansion for $f(x)$ around $x=x_m$ the  terms in the numerator of \eqref{limit_d2_phi} can be developed as follows
$$\begin{array}{rcl}
f'(x)&=&2A(x-x_m) +3B( x-x_m)^2+O((x-x_m)^3),\\\\
2\sqrt{A}\sqrt{f(x)}&=& 2A\vert x-x_m\vert\sqrt{1+\frac{B}{A}(x-x_m)+O((x-x_m))^2}\\
&=&2A\vert x-x_m\vert\big(1+\frac{B}{2A}(x-x_m)+O((x-x_m)^2)\big),\\\\
\frac{2B}{A}f(x)&=&2B(x-x_m)^2+O((x-x_m)^3).
\end{array}
$$
We substitute into \eqref{limit_d2_phi} and it holds that the terms remaining in the numerator are of order $(x-x_m)^3$.  Indeed,
$$
\begin{array}{l}
\dis\lim_{x\to x_m^-} \frac{ - f'(x)- 2\sqrt{A}\sqrt{f(x)}+\frac{2B}{A}f(x)}{4 f(x)(x-x_m)}\\\\
=\dis\lim_{x\to x_m^-} \left[\frac{-2A(x-x_m) -3B( x-x_m)^2+2A(x-x_m)(1+\frac{B}{2A}(x-x_m))}{4A(x-x_m)^3+4B(x-x_m)^4+O((x-x_m)^4)}\right.\\\\
\qquad \dis\left.+\frac{2B(x-x_m)^2+\frac{2B^2}{A}(x-x_m)^3+O((x-x_m)^3)}{4A(x-x_m)^3+4B(x-x_m)^4+O((x-x_m)^4)}\right] \\\\
=\dfrac{B^2}{2A^2},
\end{array}
$$ 
and by an analogous procedure we can obtain
$$
\dis\lim_{x\to x_m^+} \frac{ - f'(x)+ 2\sqrt{A}\sqrt{f(x)}+\frac{2B}{A}f(x)}{4 f(x)(x-x_m)}=\dfrac{B^2}{2A^2}.
$$ 
We then conclude that the second derivative of $\phi$ at point $(0,x_m)$ is bounded and
$$
\p_{xx}\phi(0,x_m)=G'(x_m)+\frac{\overline{a}_{xxx}^2(x_m)}{18\ \overline{a}_{xx}^2(x_m)}.
$$
 }

\subparagraph{{Step 3:} $\lambda_2=\sqrt{-\bar a_{xx}(x_m)/2}$.}
We next evaluate the second equation in \eqref{System_approx} at $x=x_m$ to obtain $\lambda_2=\sqrt{-\bar a_{xx}(x_m)/2}$.

\end{proof}

\section{An illustrating biological example}
\label{Bio_example}

In this section we discuss the effect of the periodic fluctuations on the critical speed of survival and the {phenotypic distribution of the population} for the following particular growth rate

\begin{equation}
\label{a_no_clim}
\SM{a(\SM{e},x)=r-g(e)(x-\theta(e))^2,}
\end{equation}
\SM{where $r$ is a positive constant corresponding to the maximal growth rate.   The positive function $g$  %and} $1-$periodic 
  represents {the pressure of   selection} and the function $\theta$ represents the optimal trait, both being functions of the   environmental state $e$. As above, we assume that  $e(t):\R_+\to E$ is a  periodic function with period $T=1$.}\\

%We introduce the effect of a climate shift on the trait $x$ as follows 
%\begin{equation}
%\label{a_with_clim}
%a(t,x-ct)=r-g(t)(x-ct-\theta(t))^2,
%\end{equation}
%where $c>0$ is the climate speed.  Note that all the computations that follow have been done with $1-$periodic functions $g(t)$ and $\theta(t)$ for simplicity but they can be generalized to every period $T$ without a great cost. \\
%We then substitute into \eqref{Case_n} after the change of variable $x'=x+ct$, it holds
%$$
%\p_t n_\e-c\p_xn_\e-\e^2\p_{xx}n_\e=n_\e\left[r-g(t)(x-\theta(t))^2-\rho_\e(t)\right];\quad
%\rho_\e(t)=\int_\R n_\e(t,x)dx.
%$$
\noindent
We compute the mean of $a(\SM{e(t)},x)$ 
$$
\overline{a}(x)=\int_0^1a(\SM{e(t)},x)dt=r-x^2\bar g+2xg_1-g_2,
$$
where
\begin{equation}
\label{g}
\bar{g}=\int_0^1g(\SM{e(t)})dt,\quad g_1=\int_0^1g(\SM{e(t)})\theta(\SM{e(t)})dt,\quad g_2=\int_0^1g(\SM{e(t)})\theta^2(\SM{e(t)})dt,
\end{equation}
and we observe that the maximum of $\overline{a}(x)$ is attained at $x_m=\frac{g_1}{\bar{g}}$, with
$$
\bar a(x_m)=r+\frac{g_1^2}{\bar{g}}-g_2.
$$
\SM{In what follows, we try to characterize the phenotypic density $n_\e$, the solution to \fer{n_eps}.}

\SM{From} {Theorem} \ref{Unique_Identif}-(ii) we obtain that $\psi(x)$ the solution of the Hamilton-Jacobi equation \eqref{LimitEq} attains its maximum at 
$$
\bar x=x_m-\frac{c}{2\sqrt{\bar{g}}}=\frac{g_1}{\bar{g}}-\frac{c}{2\sqrt{\bar{g}}}.
$$
Let $\psi(x)$ be given by \eqref{Exp_Sol_psi}, then for this specific growth rate it can be written as follows
$$
\begin{array}{rcl}
\psi(x)&=&\dis\frac{c}{2}\left(x_m-\frac{c}{2\sqrt{\bar g}}-x\right)+\int_{x_m-\frac{c}{2\sqrt{\bar g}}}^{x_m}\sqrt{\bar{g}(y-x_m)^2}dy-\left|\int_{x_m}^x\sqrt{\bar{g}(y-x_m)^2}dy\right|\\
&=&-\dis\frac{\sqrt{\bar g}}{2}\left(x+\frac{c}{2\sqrt{\bar g}}-\frac{g_1}{\bar{g}}\right)^2\\
&=&-\dis\frac{\sqrt{\bar g}}{2}\left(x-\bar{x}\right)^2.
\end{array}
$$
Moreover, the asymptotic expansions in Theorem \ref{Asymp_Expansion} {imply that}
{$$
\overline{\rho}_\e= r+\frac{g_1^2}{\bar{g}}-g_2-\dfrac{c^2}{4}-\e\sqrt{\bar{g}}+o(\e),\quad 
c_\e^*= 2\sqrt{r+\dfrac{g_1^2}{\bar{g}}-g_2}-{\e\sqrt{\frac{\bar{g}}{r+\dfrac{g_1^2}{\bar{g}}-g_2}}}+o(\e).
$$}
Furthermore, following the arguments in \cite{Mirrahimi&Figueroa_pprint}-Section 5, we can also obtain an approximation of order $\e$ for the phenotypic mean $\mu_\e$ and the variance $\sigma^2_\e$ of the population's distribution, that is:
$$
\begin{array}{rcl}
\mu_\e(t) &=&\dis\frac{1}{\rho_\e(t)}\int_{\R}x\ n_\e(t,x)dx   =\frac{g_1}{\bar{g}}-\frac{c}{2\sqrt{\bar{g}}} +\e D(t)+o(\e), 
\\\\
\sigma^2_\e &=&\dis\frac{1}{\rho_\e(t)}\int_{\R}(x-\mu_\e)^2n_\e(t,x)dx    = \frac{\e}{\sqrt{\bar{g}}}+o(\e), 
\end{array}
$$
where $D(t)=\p_x \phi(\bar x,t)$ for  $\phi$ the solution of the system \eqref{System_approx}. {We refer the readers to the Appendix B for more details on the derivation of the moments}.\\
One can verify that for this growth rate we have
$$
D(t)=-c\sqrt{\bar{g}}\left(t-\frac{1}{2}\right)+2\int_0^1\int_0^\tau g(\SM{e(s)})(\bar{x}-\theta(\SM{e(s)}))dsd\tau-2\int_0^tg(\SM{e(s)})(\bar{x}-\theta(\SM{e(s)}))ds.
$$
Note that the phenotypic mean is $1-$periodic since $D(0)=D(1)$. Moreover $\langle \mu_\e(t)\rangle=\frac{g_1}{\bar{g}}-\frac{c}{2\sqrt{\bar{g}}} {+o(\e)}$ since $\int_0^1D(t)dt=0$.\\\\
We are now interested in comparing these quantities with the case where there is no fluctuations. To do {so} we first consider a case where $\SM{g(e)}=g>0$ is constant and then a case where {$\theta$} is constant.\\

\begin{itemize}
\item[Case 1.] $\SM{g(e)}=g$ constant.  Note that, in such a case $g_1=g\bar\theta$ and $g_2=g\int_0^1\theta^2(\SM{e(t)})dt$ with $\bar \theta=\int_0^1\theta(\SM{e(t)})dt$. We compute 
\begin{equation}
\label{moments-thetap}
\begin{array}{l}
  \overline{\rho}_{\e,g(e)=g}= r+g\left[\bar{\theta}^2-\int_0^1\theta^2(\SM{e(t)})dt\right]-\dfrac{c^2}{4}-\e\sqrt{g}+o(\e),  \\\\
   \big\langle \mu_{\e,g(e)=g}(t)\big\rangle=\bar \theta-\frac{c}{2\sqrt{g}}+o(\e), \\\\
c_{\e,g(e)=g}^*= 2\sqrt{r+g\left[\bar{\theta}^2-\int_0^1\theta^2(e(t))dt\right]}-\e\sqrt{\dfrac{g}{r+g\left[\bar{\theta}^2-\int_0^1\theta^2(\SM{e(t)})dt\right]}}+o(\e).
\end{array}
\end{equation}
We compare then, the sub-cases where \SM{$e$} is constant or periodic.
\begin{itemize}
\item[a)] If $e(t)$ is a $1-$periodic function then, $\bar \theta^2<\int_0^1\theta^2(e(t))dt$  and we obtain 
$$
\begin{array}{c}
    \overline{\rho}_{\e,p}<  r-\dfrac{c^2}{4}-\e\sqrt{g}+o(\e),\quad  \big\langle \mu_{\e,p}(t)\big\rangle=\bar \theta-\dfrac{c}{2\sqrt{g}}+o(\e), \\
 c_{\e,p}^*< 2\sqrt{r}-\e\sqrt{\dfrac g r}+o(\e).
\end{array}
$$
\item[b)] If $\SM{e(t)\equiv   \overline e}$ is constant, typically equal to the averaged state of the periodic environment above, (so that \SM{$\theta\equiv\theta({\overline e})=:\theta_{\overline e}$})   we obtain in particular that $\bar \theta^2=\int_0^1\theta^2(\SM{\overline e})dt$ and hence
$$
\begin{array}{c}
     \overline{\rho}_{\e,c}= r-\dfrac{c^2}{4}-\e\sqrt{g} +o(\e),\quad \big\langle \mu_{\e,c}(t)\big\rangle=\theta_{\overline e}-\dfrac{c}{2\sqrt{g}}+o(\e),\\
     c_{\e,c}^*= 2\sqrt{r}-\e\sqrt{\dfrac g r}+o(\e).
\end{array}
$$
\end{itemize}
Thus, by keeping the pressure of selection constant, we deduce that, {for $\e$ small,}
$$
\overline{\rho}_{\e,p}\leq \overline{\rho}_{\e,c}\quad \text{ and } \quad c_{\e,p}^*\leq c_{\e,c}^*. 
$$
This means that having an oscillating optimal trait is not beneficial for the population, in the sense that the mean total size of the population decreases with respect to the case with a constant optimal trait and the critical speed which leads the population to extinct is smaller in the periodic case. Note also from \eqref{moments-thetap} that in the periodic case the mean population size is reduced by the product of the selection pressure and the variance of the optimal trait (that is $\int_0^1\theta^2(\SM{e(t)})dt-\bar{\theta}^2$).\\

\item[Case 2.] $\theta(\SM{e(t)})=\theta$ constant. Note that, in such a case $g_1=\bar g\theta$ and $g_2=\bar g\theta^2$. We compute 
$$
\begin{array}{c}
\overline{\rho}_{\e,\theta(e(t))=\theta}= r-\dfrac{c^2}{4}-\e\sqrt{\bar g} +o(\e),
\quad \big\langle \mu_{\e,\theta(e(t))=\theta}(t)\big\rangle= \theta-\dfrac{c}{2\sqrt{\bar g}} +o(\e),\\
c_{\e,\theta(e(t))=\theta}^*= 2\sqrt{r}-\e\sqrt{\dfrac{\bar g}{r}} +o(\e).
\end{array}
$$
We compare then, the sub-cases where \SM{$e$} is constant or periodic.
\begin{itemize}
\item[a)] If $\SM{e(t)}$ is a $1-$periodic function then we obtain 
$$
\begin{array}{c}
\overline{\rho}_{\e,p}= r-\dfrac{c^2}{4}-\e\sqrt{\bar g} +o(\e),\quad  \big\langle \mu_{\e,p}(t)\big\rangle= \theta-\dfrac{c}{2\sqrt{\bar g}}+o(\e),  \\
\dis c_{\e,p}^*= 2\sqrt{r}-\e\sqrt{\dfrac{\bar g}{r}}+o(\e).
\end{array}
$$
\item[b)] If $\SM{e(t)\equiv \overline e}$ is constant, typically equal to the averaged state of the periodic environment above,    (\SM{so that $g\equiv g({\overline e})=:g_{\overline e}$}), we obtain 
$$
\begin{array}{c}
\overline{\rho}_{\e,c}= r-\dfrac{c^2}{4}-\e\sqrt{g_{\overline e}} +o(\e),\quad \big\langle \mu_{\e,c}(t)\big\rangle=\theta-\dfrac{c}{2\sqrt{g_{\overline e}}}+o(\e),\\
c_{\e,c}^*= \dis 2\sqrt{r}-\e\sqrt{\f{g_{\overline e}}{r}}+o(\e).
\end{array}
$$
\end{itemize}
If we choose an oscillating selection pressure function $\SM{g(e)}$ which satisfies:
\begin{equation}
\label{Cond_g}
\bar g < \SM{g(\overline e)},
\end{equation}
which holds for instance if $g(\cdot)$ is a concave function, then we obtain that
$$
\overline{\rho}_{\e,c} < \overline{\rho}_{\e,p}\quad \text{ and } \quad c_{\e,c}^* < c_{\e,p}^*. 
$$
This means that the mean total size of the population increases with respect to the case with a constant \SM{environmental state}. Moreover, the critical speed above which the population goes extinct is larger in the periodic case. This means that the periodic fluctuations can help the population to follow the environment change. \\
Note that the condition \eqref{Cond_g} imposed to \SM{$g(e(t))$} is the opposite to the one imposed in  \cite{Mirrahimi&Figueroa_pprint} (equation 51 of Section 6.2), leading to more performant populations. There,  it was proved that in presence of the mutations and while the fluctuations act on the pressure of the selection (that is with a similar growth rate, however with $c=0$ and under the condition $\bar g>g(\SM{\overline e })$), a fluctuating environment can select for a population with smaller variance and in this way lead to more performant populations. What is beneficial in a (in average) constant environment may indeed be disadvantageous in a changing environment. \\
Note also that in the present example under condition \eqref{Cond_g}, we have
$$
\Big\vert\big\langle \mu_{\e,p}(t)\big\rangle-\theta\Big\vert > \Big\vert \big\langle \mu_{\e,c}(t)\big\rangle-\theta\Big\vert.
$$
This means that even if the population can follow the \SM{environmental} change in a better way by considering a fluctuating environment, this population is \SM{less   adapted}.
\end{itemize}

\appendix
\section{The proofs of some regularity estimates}
\subsection{Uniform bounds for $\rho_\e$: the proof of Proposition \ref{bounds_rhoe}.}
\begin{proof} \\
From equation \eqref{n_eps} integrating in $x\in\R$ and using assumption \eqref{a_W3inf} we obtain:
\begin{equation}
\label{bornDerRho}
\frac{d\rho_\varepsilon}{dt}=\int_{\R}n_\varepsilon(t,x)[a(\SM{e(t)},x)-\rho_\varepsilon(t)]dx\leq \rho_\varepsilon(t)[d_0-\rho_\varepsilon(t)].
\end{equation}
This implies that 
$$
\rho_\varepsilon(t)\leq\rho_M:=\max(\rho_\varepsilon^0,d_0).
$$
For the lower bound we use the explicit expression \eqref{Exp_rho_eps}  for $\rho_\e$, the solution of \eqref{Rho_equation}.  We come back to equation \eqref{Q_eps&lambda_eps}, which gives, thanks to \eqref{a_W3inf}, \eqref{lambda_m} and \eqref{Exp_rho_eps}  the following lower bound for $\rho_\e$
$$
0< \rho_m:=\frac{1}{T}e^{-d_0T}\left(e^{\lambda_mT}-1\right)\leq \rho_\varepsilon(t),\quad \forall\;t\geq0.
$$
\end{proof}
\subsection{Upper bound for $\psi_\e$: the proof of the r.h.s of \eqref{Bounds_psi_eps}}
We prove that $\psi_\e$ is bounded from above using the equation for $n_\e$.
From \eqref{bound_p}, we have for $p_{c\e}$:
\begin{equation}
\label{bound_p_eps}
p_{c\e}(t,x)\leq\Vert p_{c\e}\Vert_{L^\infty}e^{-\frac{1}{\e}\left(-\frac{c}{2}+\sqrt{\delta+\frac{c^2}{2}}\right)(|x|-R_0)},\quad\forall(t,x)\in[0,+\infty)\times\R.
\end{equation}
Define $q_{c\e}(t,x)=p_{c\e}(t,x\e)$, which satisfies 
\begin{equation}
\label{q_eps}
\left\{
\begin{array}{cr}
\partial_t q_{c\e}-c\p_x q_{c\e}-\p_{xx} q_{c\e}=a_\e(\SM{e(t)},x)q_{c\e},&\mathrm{in}\;[0,+\infty)\times \R,\\
0<q_{c\e}(t,x)=q_{c\e}(t+ T,x)
\end{array}
\right.
\end{equation}
for $a_\e(\SM{e(t)},x)=a(\SM{e(t)},x\e)+\lambda_{c,\e}$. Note that $a_\e$ is uniformly bounded thanks to the $L^\infty-$norm of $a$. Moreover we have the following bounds for $\lambda_{c,\e}$ coming from \eqref{Q_eps&lambda_eps},
\begin{equation}
\label{Bound_lambda_ce}
-d_0\leq\lambda_{c,\e}\leq-\lambda_m.
\end{equation}
We recall that $p_{c\e}$ is uniquely determined once $\Vert p_{c\e}(0,x)\Vert_{L^\infty(\R)}=1$ is fixed. Then, one can choose $x_\e$ such that $ p_{c\e}(0,x_\e)=1$. Note also that $q_\e$ is a nonnegative solution of \eqref{q_eps} in the bounded domain $(0,2T)\times B(\frac{x_\e}{\e},1)$. \\
Here we apply an elliptic-type Harnack inequality for positive solutions of \eqref{q_eps} in a bounded domain, (see for instance Theorem 2.5 \cite{huska06}). Let $\delta_0$, be such that $0<\delta_0<T$, then we have: 
$$
\sup_{x\in B(\frac{x_\e}{\e},1)}q_{c\e}(t,x)\leq C\inf_{x\in B(\frac{x_\e}{\e},1)}q_{c\e}(t,x),\quad \forall\;t\in[\delta_0,2T],
$$
where $C$ is a positive constant depending on $\delta_0$ and $d_0$. Coming back to $p_{c\e}$ this implies
\begin{equation}
\label{BoundInf_p_App}
p_{c\e}(t_0,x_\e)\leq \sup_{y\in B(x_\e,\e)}p_{c\e}(t_0,y)\leq C p_{c\e}(t_0,x),\quad \forall(t_0,x)\in[\delta_0,2T]\times B(x_\e,\e).
\end{equation}
And we use the $T-$periodicity of $p_{c\e}$ to conclude that the last inequality is satisfied for $t\in[0,T]$.\\
From  \eqref{Rho_bouded}, \eqref{bound_p_eps} and \eqref{BoundInf_p_App} we obtain 
$$
n_\e(0,x)\leq \rho_M \dfrac{p_{c\e}(0,x)}{\int_{\R}p_{c\e}(0,x)dx}\leq \rho_M\frac{Cp_{c\e}(0,x)}{\int_{B(x_\e,\e)}p_{c\e}(0,x_\e)dx}= \rho_M \dfrac{Cp_{c\e}(0,x)}{|B(x_\e,\e)|}\leq C'\e^{-1}e^{\frac{c_1-c_2|x|}{\e}},
$$
for all $\varepsilon\leq\varepsilon_0$, with $\varepsilon_0$ small enough, where the constant $c_1$ depends on $\rho_M$, $\delta$, $R_0$ and $c$, and $c_2=-\frac{c}{2}+\sqrt{\delta+\frac{c^2}{2}}$. Next we proceed with a maximum principle argument to obtain for every $(t,x)\in[0,+\infty)\times\R$ and $c_3=c_2(c+c_2)+d_0$, 
$$
n_\varepsilon(t,x)\leq C'e^{\frac{c_1-c_2|x|}{\varepsilon}+c_3 t}.
$$
From the latter inequality and the periodicity of $\psi_\e$, with an abuse of notation for constant $c_1$, we deduce that:
\begin{equation}
\label{Up_Bound_psi_e}
\psi_\e(t,x)\leq c_1-c_2|x|,\quad \forall(t,x)\in[0,+\infty)\times\R.
\end{equation}
\subsection{Equicontinuity in time for $\psi_\e$.}
We will use the arguments in \cite{Mirrahimi&Figueroa_pprint}, which follow a method introduced in \cite{BarlesEqui}, in order to deduce uniform equicontinuity in time for the family $\psi_\e$ on compact subsets of $(0,+\infty)\times\R$.\\
The goal will be to find for any $\eta>0$, constants $\Lambda_1$, $\Lambda_2$ large enough such that: for any $x \in B(0, R/2)$, $s\in[0,T]$, and for all $\varepsilon<\varepsilon_0$ we have
\begin{equation}
\label{supersolution}
\psi_\e(t,y)-\psi_\e(s,x)\leq \eta + \Lambda_1|x-y|^2+\varepsilon \Lambda_2(t-s), \forall (t,y)\in[s,T]\times B_R(0),
\end{equation}
and
\begin{equation}
\label{subsolution}
\psi_\e(t,y)-\psi_\e(s,x)\geq -\eta- \Lambda_1|x-y|^2-\varepsilon \Lambda_2(t-s), \forall (t,y)\in[s,T]\times B_R(0).
\end{equation}
Because of the analogy between the both inequalities above we only prove \eqref{supersolution}. 
\\We fix $(s, x)$ in $[0, T[\times B_{R/2}(0)$ and define
$$
\widehat{\xi}(t,y)=\psi_\e(s,x) +\eta+ \Lambda_1|x-y|^2+\varepsilon \Lambda_2(t-s), \quad (t,y)\in[s,T[\times B_R(0),
$$
with $\Lambda_1$ and $\Lambda_2$ positive constants to be determined. We prove that, for $\Lambda_1$ and $\Lambda_2$ large enough, $\widehat{\xi}$ is a super-solution of the equation \eqref{Eq_psi_eps} on $[s,T]\times B_R(0)$ and $\widehat{\xi}(t,y)>\psi_\e(t,y)$ for $(t,y)\in \{s\}\times B_R(0)\cup [s,T]\times \partial B_R(0)$.\\\\
According to Section 3.2.1, $\{\psi_\e\}_\e$ is locally uniformly bounded, so we can find a constant $\Lambda_1$ such that for all $\e < \e_0$,
$$
\frac{8\Vert \psi_\e\Vert_{L^\infty([0,T]\times B_R(0))}}{R^2}\leq \Lambda_1. 
$$
With this choice, $\widehat{\xi}(t, y) > \psi_\e(t, y)$ on $[s, T ]\times \partial B_R(0)$, for all $\eta>0$, $\Lambda_2>0$ and $x\in B_{R/2}(0)$. \\
Next we prove that, for $\Lambda_1$ large enough, $\widehat{\xi}(s, y) > \psi_\e(s, y)$ for all $y\in B_R(0)$. We
argue by contradiction. Assume that there exists $\eta > 0$ such that for every constant $\Lambda_1$ there exists $y_{\Lambda_1,\varepsilon}\in B_R(0)$ such that
\begin{equation}
\label{Absurd}
\psi_\e(s,y_{\Lambda_1,\varepsilon)}-\psi_\e(s,x)>\eta+\Lambda_1|y_{\Lambda_1,\varepsilon}-x|^2.
\end{equation}
This implies
$$
|y_{\Lambda_1,\varepsilon}-x|\leq\sqrt{\frac{2\Psi_M}{\Lambda_1}}\longrightarrow0,\quad\mathrm{as}\; {\Lambda_1\rightarrow\infty},
$$
where we have denoted $\Psi_M$  a uniform upper bound for $\Vert \psi_\e\Vert_{L^\infty([0,T]\times B_R(0))}$. Then for all $\delta_1>0$, there exist $\Lambda_1$ large enough and $\varepsilon_0$ small enough, such that $\forall\varepsilon<\varepsilon_0$,
$$
|y_{\Lambda_1,\varepsilon}-x|\leq \delta_1.
$$
Therefore, from the uniform continuity in space of $\psi_\e$ taking  $\delta_1$ small enough, we obtain
$$
|\psi_\e(s,y_{\Lambda_1,\varepsilon})-\psi_\e(s,x)|<\eta/2\quad \forall\varepsilon\leq \varepsilon_0,
$$
but this is a contradiction with \eqref{Absurd}. Therefore $\widehat{\xi}(s, y) > \psi_\e(s, y)$ for all $y\in B_R(0)$. \\
Finally, noting that $R<+\infty$ we deduce that for $\Lambda_2$ large enough, $\widehat{\xi}$ is a super-solution to \eqref{Eq_psi_eps} in $[s, T ] \times B_R(0)$. \\
Using a comparison principle, we have
$$
\psi_\e(t,y)\leq \widehat{\xi}(t,y)\quad \forall(t,y)\in[s,T]\times B_R(0).
$$
Thus \eqref{supersolution} is satisfied for $t\geq s\geq 0$. To conclude we put $x = y$ and obtain that for all $\eta>0$ there exists $\varepsilon_0>0$ such that for all $\e<\e_0$
$$
|\psi_\e(t,x)-\psi_\e(s,x)|\leq \eta + \varepsilon \Lambda_2(t-s),
$$
for every $(t,x)\in[0,T]\times B_R(0)$.
This implies that $\psi_\e$ is locally equicontinuous in time. Moreover we obtain that
\begin{equation}
\label{Equi_psi}
\forall R>0,\ \sup_{t\in[0,T],\ x\in B_R}|\psi_\e(t,x)-\psi_\e(s,x)|\rightarrow0\quad\mathrm{as }\:\e\rightarrow 0.
\end{equation}

\section{Computations of the moments of the population's distribution}
\label{Moments}

In this section we estimate the moments of the population's distribution with a small error, {following \cite{Mirrahimi2017, Mirrahimi&Figueroa_pprint}}. To this end, we use the arguments in Section \ref{Eigen_Approx}.\\
Using \eqref{Exp_Sol_psi}, one can compute a Taylor expansion of order 4 around the point of maximum $\overline{x}$.
\begin{equation}
\label{ExpTay_psi}
\psi(x)=-\frac{A}{2}(x-\bar x)^2+B(x-\bar x)^3+C(x-\bar x)^4+o(x-\bar x)^4.
\end{equation}
Note also that using the formal expansions in \eqref{Formal_exp_phi_lambda} one can obtain $\phi$ from an equivalent system to \eqref{System_approx}, that is
\begin{equation}
\label{System_approx_App}
\left\{
\begin{array}{l}
\dis\p_t \phi=\left|\psi_x +\frac{c}{2}\right|^2+a(\SM{e(t)},x)-\frac{c^2}{4}-\bar{\rho},\\\\
\dis-\psi_{xx}=[2\psi_x +c]{\frac{1}{T}\int_0^T\p_x \phi(t,x)dt}-\psi_{xx}(x_m).
\end{array}
\right.
\end{equation}
and write (formally)
$$
\phi(t,x)=\phi(t,\bar x)+D(t)(x-\bar x)+E(t)(x-\bar x)^2+o(x-\bar x)^2,\quad \omega(t,x)=F(t)+o(x-\bar x).
$$
The above approximations of $\psi$, $\phi$ and $\omega$ around the maximum point of $\psi$ allow us to estimate the moments
of the population's distribution with an error of order $\e^2$ as $\e\rightarrow 0$. Indeed, we use the following  approximation for the phenotypic density of the population
$$
{n_\e(t,x)=\frac{1}{\sqrt{2\pi\e}}e^{\frac{\psi(x)}{\e}+\phi(t,x)+O(\e)},}
$$ 
and replacing by the approximations of $\psi$, $\phi$ and $\omega$ given above, we can obtain
$$
\begin{array}{lcl}
\dis\int_{\R}(x-\bar x)^kn_\e(t,x)dx&=&\dfrac{e^{\phi(t,\bar x)}\e^{\frac{k}{2}}}{\sqrt{2\pi}}\dis\int_{\R}\left[y^k e^{\frac{-Ay^2}{2}}\big[1+\sqrt{\e}\left(By^3+D(t)y\right)\right.\\\\ && \left.\dis\qquad +\e\left(Cy^4+E(t)y^2+F(t)+\frac{1}{2}(By^3+D(t)y)^2\right)+o(\e)\big]dy\right].
\end{array}
$$
Note that, {we performed above} a change of variable $x-\bar x=\sqrt{\e}\ y$. Therefore each term $x-\bar x$ can be considered as of order $\sqrt{\e}$ in the integration. The above computation leads in particular to the following approximations of the {mean phenotypic trait and the variance:}
$$
\begin{array}{l}
\mu_\e(t)=\dis\frac{1}{\rho_\e(t)}\int_{\R}x\ n_\e(t,x)dx= \bar x+\e\left(\frac{3B}{A^2}+\frac{D(t)}{A}\right)+O(\e^2),\\
\sigma^2_\e(t)=\dis\frac{1}{\rho_\e(t)}\int_{\R}(x-\mu_\e)^2n_\e(t,x)dx=\frac{\varepsilon}{A}+O(\e^2).
\end{array}
$$
\subsection{Application to the biological example.}
We now apply the previous computations to the particular growth rate  $a$ given in \eqref{a_no_clim}.
{Thanks to \eqref{LimitEq}, we obtain }
$$
\overline{\rho}=\overline{a}(\bar x)=r-\dfrac{g_1^2}{\bar{g}}-\dfrac{c^2}{4}-g_2,
$$
for $\bar g$, $g_1$ and $g_2$ given in \eqref{g}. Next {from the second equation of \eqref{System_approx_App} and the fact that $\psi_{xx}(x)=\psi_{xx}(x_m)$, $\forall\ x\in\R$, we obtain that}
$$
0=\left(\int_0^1\p_x\phi (t,x)dt\right)\left[-2\sqrt{\bar{g}}\left(x-\bar{x}\right)+c\right],\quad {\forall\ x\in\R,}
$$
which implies directly,  that
\begin{equation}
\label{Moyenne_null}
\int_0^1\p_x \phi(t,x)dt=0.
\end{equation}
On the other hand, by substituting $\psi(x)$ in \eqref{ExpTay_psi} we find $A=\sqrt{g}$, $B=C=0$.\\
Moreover, we differentiate with respect to $x$ in  the first equation of \eqref{System_approx_App}  and obtain:
$$
\p_x\p_t\phi=\p_x\big(a(\SM{e(t)},x)-\bar a(x)\big)=2x(\bar{g}-g(\SM{e(t)}))+2g(\SM{e(t)})\theta(\SM{e(t)})-2g_1,
$$
and an integration in $[0,t]$ gives
$$
\p_x\phi(t,x)-\p_x\phi(0,x)=2t(x\bar{g}-g_1)-2\int_0^tg(\SM{e(s)})(x-\theta(\SM{e(s)}))ds.
$$
We next integrate the last equality in $\left[0,1\right]$ and use \eqref{Moyenne_null} to obtain 
$$
\p_x \phi(0,x)=-x\bar{g}+g_1+2\int_0^1\int_0^tg(\SM{e(s)})(x-\theta(\SM{e(s)}))dsdt,
$$
from where we deduce that
$$
\begin{array}{rl}
\p_x\phi(t,x)&=(2t-1)(x\bar{g}-g_1)+2\int_0^1\int_0^\tau g({e(s)})(x-\theta({e(s)}))dsd\tau\\
&-2\int_0^tg({e(s)})(x-\theta({e(s)}))ds.
\end{array}
$$
Evaluating in $\bar{x}$ gives
$$
\begin{array}{rl}
   D(t)=\p_x\phi(t,\bar{x})=&-c\sqrt{\bar{g}}\left(t-\frac{1}{2}\right)+2\int_0^1\int_0^\tau g(\SM{e(s)})(\bar{x}-\theta(\SM{e(s)}))dsd\tau  \\\\
&-2\int_0^tg(\SM{e(s)})(\bar{x}-\theta(\SM{e(s)}))ds.
\end{array}
$$
Now we are able to give an approximation of the population mean size $\overline{\rho}_\e$, the phenotypical mean $\mu_\e$ and the variance $\sigma^2_\e$ of the population's distribution, following the previous computations, that is:
$$
 \mu_\e(t)\approx\frac{g_1}{\bar{g}}-\frac{c}{2\sqrt{\bar{g}}} +\e D(t) 
,\quad \sigma^2_\e\approx \frac{\e}{\sqrt{\bar{g}}}.
$$

\bigskip

 {\bf Acknowledgments.}
The authors thank Jean-Marc Bouclet for fruitful discussions. The second author is also grateful for partial funding from the European Research Council (ERC) under the European Union's Horizon 2020 research and innovation programme (grant agreement No 639638), held by Vincent Calvez, and {from the chaire Mod\'elisation Math\'ematique et Biodiversit\'e of V\'eolia Environment-\'Ecole Polytechnique -Museum National d'Histoire Naturelle- Fondation X. }

\bibliographystyle{plain}	
\bibliography{biblio}

\end{document}